\documentclass[12pt, final]{article}
\usepackage{a4}
\usepackage{amsmath}%
\usepackage{amstext}%
\usepackage{amssymb}%
\usepackage{amsxtra}
\usepackage{graphicx}
\usepackage{showkeys}%
\usepackage{epsfig}%
\usepackage{cite}
\usepackage{tikz}
\usepackage{caption}
\usepackage{longtable}
\usepackage{threeparttable,booktabs,multirow,lscape}
\usepackage{lipsum}
\usepackage{rotating}
\usepackage{multirow}
\usepackage{booktabs}
\usepackage{floatrow}
\usepackage{amsmath}
\usepackage{amssymb}
\usepackage{graphicx}
\usepackage[colorlinks,urlcolor=red]{hyperref}
\usepackage{graphicx}
\usepackage{subfigure}
\usepackage{listings}
\usepackage[ruled,vlined]{algorithm2e}

\newtheorem{theorem}{Theorem}

\newtheorem{lemma}[theorem]{Lemma}

\newtheorem{proposition}[theorem]{Proposition}
\newtheorem{rem}[theorem]{Remark}
\newenvironment{remark}{\rem\rm}{\endrem}

\newcounter{unnumber}

\newtheorem{prf}[unnumber]{Proof.}
\newenvironment{proof}{\prf\rm}{\hfill{$\blacksquare$}\endprf}
\newcommand{\R}{\mathbb{R}}%
\newcommand{\N}{\mathbb{N}}%
\newcommand{\prox}{\mathrm{prox}}%

\renewcommand{\>}{\right\rangle}
\newcommand{\<}{\left\langle}

\DeclareMathOperator*\dist{dist}%
\DeclareMathOperator*\ran{ran}%
\DeclareMathOperator*\argmin{argmin}

\title{Incremental proximal gradient scheme with penalization for constrained composite convex optimization problems}

\author{Narin Petrot\thanks{Department of Mathematics, Faculty of Science, Naresuan University,	Phitsanulok,  65000, Thailand, Center of Excellence in Nonlinear Analysis and Optimization, Faculty of Science, Naresuan University,
		Phitsanulok,  65000, Thailand, 
email: narinp@nu.ac.th.}
\and	
Nimit Nimana\thanks{Department of Mathematics, Faculty of Science, Khon Kaen University, Khon Kaen, 40002 Thailand,
	email: nimitni@kku.ac.th.}%
}

\begin{document}
		\maketitle
	\begin{abstract}
		We consider the problem of minimizing a finite sum of convex functions subject to the set of minimizers of a convex differentiable  function. In order to solve the problem, an algorithm  combining the incremental proximal gradient method with smooth  penalization technique is proposed. We show the convergence of the generated sequence of iterates to an optimal solution of the optimization problems, provided that a condition expressed via the Fenchel conjugate of the constraint function is fulfilled.  Finally, the functionality of the method is illustrated by some numerical experiments addressing  image inpainting problems and generalized Heron problems with least squares constraints.
		
		
	\end{abstract}


		\section{Introduction}
		
 Let $F_i:\R^n\to\R$ be a function of the form
		$$F_i(x):=f_i(x)+h_i(x)$$ 
		for all  $i=1,\ldots,m$, where $f_i:\R^n\to\R$ is a convex function and $h_i:\R^n\to\mathbb{R}$ is a convex differentiable function such that $\nabla h_i$ is $L_i-$Lipschitz continuous. Let $g:\R^n\to\R$ be a convex differentiable function such that $\nabla g$ is $L_g-$Lipschitz continuous.  In this work, we focus on the problem
		\begin{eqnarray}\label{MP}%
			\begin{array}{ll}
				\textrm{minimize}\indent \sum_{i=1}^m F_i(x)\\
				\textrm{subject to}\indent x\in \arg\min g.
			\end{array}%
		\end{eqnarray}
		Let $\mathcal{S}$ denote the solution set of this problem and assume that $\mathcal{S}$ is nonempty.  In addition, we may assume without loss of generality that $\min g=0$.

		It is well known that the minimization of the sum of composite functions yields many applications to classification and regression models in machine learning. In these applications, a key feature is to deal with  a very large number of component (typically, convex and Lipschitz continuous) loss functions where the evaluation of  the proximal operators and/or gradients of the whole objective function seems very costly or even impossible; see \cite{BCN16,BL03,SLB17}.  Apart from the aforementioned classification and regression problems, the problems with additive structure also arise in sensor, wireless and peer-to-peer networks in which there is no central node that facilitates computation and communication. Moreover, the allocation of all the cost components $F_i$ at one node is sometimes not possible due to memory, computational power, or private information. For further discussion concerning sensor networks, see \cite{RN05, BHG08}.

		One of promising algorithms for performing this kind of problem structure is the so-called incremental type method. Its key idea  is to take steps subsequently  along the proximal operators and/or gradients of the component functions $F_i$  and to update the current iterate after processing each $F_i$. 	To be precise, let us recall the  classical incremental gradient method (IGM) for solving the minimization problem, that is,
			\begin{eqnarray}\label{MP-smooth}%
		\begin{array}{ll}
		\textrm{minimize}\indent \sum_{i=1}^m f_i(x)\\
		\textrm{subject to}\indent x\in X,
		\end{array}%
		\end{eqnarray}
	where	$f_i:\R^n\to\mathbb{R}$ is a convex differentiable function, for all  $i=1,\ldots,m$, and $X\subset\R^n$ is a nonempty closed convex set. The method is given as follows: if $x_k$ is the vector obtained after $k$ cycles, the vector $x_{k+1}$  is updated by  $\varphi_{1,k}:=x_k$, then computing
\begin{eqnarray}\label{IGM}\varphi_{i+1,k}:=\varphi_{i,k}-\alpha_k\nabla f_i(\varphi_{i,k}),\indent i=1,\ldots,m,
\end{eqnarray}
and finally generating  $x_{k+1}$  after one more cyclic $m$ steps as 
$$x_{k+1}:=\mathrm{proj}_{X}(\varphi_{m+1,k}),$$
where  $\alpha_k$ is a positive scalar parameter, and $\mathrm{proj}_{X}$ is the projection operator onto $X$. The advantage of IGM comparing with the classical gradient descent method, which have been analytically proved and even experimentally observed, is that it can attain a better asymptotic convergence to a solution of (\ref{MP-smooth}); see \cite{B12} for more details on this topic.

		Apart from the gradient based method,  there are many situations in which the objective functions  may not be smooth enough to apply IGM; in this case, we can consider the so-called incremental proximal method instead.  Consider the (nonsmooth) minimization problem (\ref{MP-smooth}) where the component function	$f_i:\R^n\to\R$ is convex for all  $i=1,\ldots,m$. The  incremental proximal method, which was initially proposed by Bertsekas \cite{B11}, is given as follows: if $x_k$ is the vector obtained after $k$ cycles, then the vector $x_{k+1}$  is updated in a similar fashion to IGM, except that the gradient step (\ref{IGM}) is replaced by the proximal step
		$$\varphi_{i+1,k}:=\prox_{\alpha_k f_i}\left(\varphi_{i,k}\right),\indent i=1,\ldots,m.$$
		For further discussion convergence result, see \cite{B11,B12,B15}.

		On the other hand,  the problem (\ref{MP-smooth})  involves  the constraints which  can be reformulated into the form (\ref{MP}) via a penalty function corresponding to the  constraint so that the set of all minimizers of the constructed penalty function is the considered constraint.  Attouch and Czarnecki \cite{AC10} initially investigated  a qualitative analysis of the optimal solutions of (\ref{MP})	from the perspective of a penalty-based dynamical system.  This  starting point stimulates  huge interest among research community to design and develop  numerical algorithms for solving the minimization problem  (\ref{MP}); see \cite{AC10, ACP11,ACP11-2,BB15,BC14,BCN17,BCN17-2,BN18,NP13,NP18, NP18-2,P12} for more insights into this research topic. It is worth noting that the common key feature of these proposed iterative methods is the penalization strategy, that is, if the function $g$ is smooth,  then the penalization term is evaluated by its gradient \cite{NP13,NP18,P12}.

Motivated by all the results mentioned above, we proposed an iterative scheme, which combines the incremental proximal and gradient method with penalization strategy, for solving the constrained minimization problem (\ref{MP}).
	 To deal with  the convergence result, we show that the generated sequence converges to an optimal solution of (\ref{MP}) by using the quasi-Fej\'er monotonicity technique. To illustrate the
	 theoretical results,  we also present  some numerical experiments addressing  the image reconstruction problem and the generalized Heron location  problem.

	In the remaining of this section we recall some necessary tools of convex analysis. The reader may consult \cite{BC11,BV10,Z02} for further details. For a convex function $f:\R^n\rightarrow\R$ we  let $f^*$ donote the (Fenchel) conjugate function of $f$, that is, the function $f^*:\R^n\rightarrow (-\infty,+\infty]$ such that  $$f^*(u):=\sup_{x\in \R^n}\{\langle u,x\rangle-f(x)\}$$ for all $u\in \R^n$.
		The subdifferential of $f$ at $x\in \R^n$, is the set
	$$\partial f(x):=\{v\in \R^n:f(y)\geq f(x)+\langle v,y-x\rangle \ \forall y\in \R^n\}.$$ 
We also let $\min f := \inf_{x \in \R^n} f(x)$ denote the optimal objective value of the function $f$ and let
$\arg\min f :=\{x \in \R^n: f(x) = \min f \}$ denote its set of global minima of $f$. 
	For $r>0$ and $x\in \R^n$, we let $\prox_{rf}(x)$ denote the proximal point of parameter $r$ of $f$ at $x$, which is the unique optimal solution of the (strongly convex) optimization problem
	$$\min_{u\in \R^n}f(u)+\frac{1}{2r}\|u-x\|^2.$$
	Note that $\prox_{rf}=(I+r\partial f)^{-1}$ and it is a single-valued operator. 

		Let $X\subseteq \R^n$ be a nonempty set. The indicator function of $X$  is the function $\delta_X:\R^n\rightarrow (-\infty,+\infty]$
		 which takes the value $0$ on $X$ and $+\infty$ otherwise. The subdifferential of the indicator function is the normal cone of $X$, that is,
		  $$N_X(x)=\{u\in \R^n:\langle u,y-x\rangle\leq 0 \ \forall y\in X\}$$ if $x\in X$ and
		$N_X(x)=\emptyset$ for $x\notin X$. For all $x\in X$, it holds that $u\in N_X(x)$ if and only if $\sigma_X(u)=\langle u,x\rangle$,
		where $\sigma_X : \R^n \rightarrow (-\infty,+\infty]$ is the support function of $X$ defined by  $$ \sigma_X(u)=\sup_{y\in X}\langle y,u\rangle.$$ Moreover, we let $\ran(N_X)$ denote the range
		of the normal cone $N_X$, that is, we have $p \in \ran(N_X)$ if and only if there exists $x \in X$
		such that $p \in N_X(x)$.

		\section{Algorithm and Convergence Result}

			In this section, we consider the convergence analysis of the  incremental proximal gradient method with smooth penalty term for solving (\ref{MP}). Firstly, we propose our main algorithm as shown in Algorithm \ref{algorithm-ergodic-smooth}.
			\vspace{0.3cm}
			
			\begin{algorithm}[H]\label{algorithm-ergodic-smooth}
				\SetAlgoLined
				\textbf{Initialization}: The positive sequences $(\alpha_k)_{k\geq1}$, $(\beta_k)_{k\geq1}$, and an arbitrary $x_1\in \R^n$. \\
				\textbf{Iterative Step}: For a given current iterate $x_{k}\in \R^n$ ($k\geq 1$), set 
				$$\varphi_{1,k}:=x_{k}-\alpha_k\beta_k\nabla g(x_k),$$
				and define 
				$$\varphi_{i+1,k}:=\prox_{\alpha_k f_i}\left(\varphi_{i,k}-\alpha_k\nabla h_i(\varphi_{i,k})\right),\hspace{1cm} i=1,\ldots,m,$$
				and 
				$$x_{k+1}=\varphi_{m+1,k}.$$
				
				\caption{IPGM with penalty term}
			\end{algorithm}
			\begin{remark}Algorithm \ref{algorithm-ergodic-smooth}  is different from \cite[Algorithm 3.1]{NP18-2}. In fact, in \cite{NP18-2}, the authors considered the problem (\ref{MP}) in the sense that $h=\sum_{i=1}^mh_i$ and performed the gradient $\nabla h(x_k)$ at each iteration $k$. However, the iterative scheme proposed here allows us to perform the gradient $\nabla h_i(\varphi_{i,k})$ at each sub-iteration $i$. Moreover, it is worth noting that Algorithm \ref{algorithm-ergodic-smooth} is very useful through its decentralized  setting which appears in many situations, for instance, decentralized network system or support vector machine learning problems; see \cite{NO09,RNV12}. 
		\end{remark}
		\vspace{0.3cm}

			
			For the convergence result,  the following hypotheses are assumed throughout this work:

			\begin{equation*}\label{H}
			\left\{		\begin{array}{ll} 
			\text{(H1) The subdifferential sum } \partial\left(\sum_{i=1}^m  f_i+\delta_{\arg\min g}\right)=\sum_{i=1}^m\partial f_i+N_{\arg\min g} \text{ holds};\\
            \text{(H2)	The sequence } (\alpha_k)_{k\geq1} \text{ is satisfying } \sum_{k=1}^{\infty}\alpha_k=+\infty \text{ and } \sum_{k=1}^{\infty}\alpha_k^2<+\infty;\\
			\text{(H3) } 0<\liminf_{k\to+\infty}\alpha_k\beta_k\leq\limsup_{k\to+\infty}\alpha_k\beta_k<\frac{2}{L_g}; \\
			\text{(H4) For all } p\in\mathrm{ran}(N_{\arg\min g}),  \sum_{k=1}^\infty\alpha_k\beta_k\left[g^*\left(\frac{p}{\beta_k}\right)-\sigma_{\arg\min g}\left(\frac{p}{\beta_k}\right)\right]<+\infty.
			\end{array}
			\right. 
			\end{equation*}%
			
				Some remarks relating to our assumptions are as follows.
		
		\begin{remark}\label{key-remark}
			\begin{itemize}
				\item[(i)] For the conditions which guarantee the  exact subdifferential sum formula in the condition (H1),  the reader may consult the book of Bauschke and Combettes \cite{BC11}.
				\item[(ii)] Note that the hypothesis (H3) is a relaxation of Assumption 4.1 (S3) in  \cite{NP18-2}. In fact, the superior limit in  \cite{NP18-2} is bounded above by $\frac{1}{L_g}$, but in this work it can be extended to $\frac{2}{L_g}$. This allows us to  consider  larger parameters $(\alpha_k)_{k\geq1}$ and $(\beta_k)_{k\geq1}$. An example of the sequences $(\alpha_k)_{k\geq1}$ and $(\beta_k)_{k\geq1}$ satisfying the conditions (H2) and (H3) is the real sequences $\alpha_k\sim\frac{1}{k}$ and $\beta_k\sim bk$ for every $k\geq1$, where $0<b<\frac{2}{L_g}$.
				\item[(iii)] The condition  (H4) was originated by  Attouch and Czarnecki \cite{AC10}. 
				For example,  for a function $g:\R^n\to\R$, it holds that $g\leq \delta_{\arg\min g}$ and so $g^*\geq (\delta_{\arg\min g})^*=\sigma_{\arg\min g}$, which yields $$g^*-\sigma_{\arg\min g}\geq 0.$$ 
				Note that if the function $g$ satisfies
				$$g \geq \frac{a}{2}\mathrm{dist}^2(\cdot,\arg\min g),$$
				where $a>0$, then we have $g^*(x)-\sigma_{\arg\min g}(x)\leq\frac{1}{2a}\|x\|^2$ for all $x \in \R^n$. Thus, for every $k \geq 1$, and all $p\in\mathrm{ran}(N_{\arg\min g})$, we have
				$$\alpha_k\beta_k\left[g^*\left(\frac{p}{\beta_k}\right)-\sigma_{\arg\min g}\left(\frac{p}{\beta_k}\right)\right]\leq \frac{\alpha_k}{2a\beta_k}\|p\|^2.$$
				Note that if $\sum_{k=1}^\infty\frac{1}{\beta_k^2}<+\infty$, then it follows that
				$$\sum_{k=1}^\infty\alpha_k\beta_k\left[g^*\left(\frac{p}{\beta_k}\right)-\sigma_{\arg\min g}\left(\frac{p}{\beta_k}\right)\right]<+\infty.$$ This inequality also holds for the sequences  satisfying the hypotheses (H2) and (H3).
			\end{itemize}
		\end{remark}
		\vspace{0.3cm}
	
					The following theorem describes the convergence of iterates.
				\vspace{0.3cm}
				
				\begin{theorem}\label{thm-main-2}
					The sequence  $(x_k)_{k\geq1}$ converges to a point in $\mathcal{S}$. 
				\end{theorem}
				\vspace{0.3cm}	
				
				In order to prove Theorem \ref{thm-main-2}, we need to recall the concept of quasi-Fej\'{e}r monotone as follows.	Let $C$ be a nonempty subset of $\R^n$. We say that a sequence $(x_k)_{k\geq1}\subset \R^n$ is   \textit{quasi-Fej\'{e}r monotone} relative  to $C$ if for each $c\in C$, there exist a sequence $(\delta_k)_{k\geq1}\subset [0,+\infty)$ with $\sum_{k=1}^{\infty}\delta_k=+\infty$ and $k_0\in\N$ such that
				$$\|x_{k+1}-c\|^2\leq\|x_k-c\|^2+\delta_k,\indent\forall k\geq k_0.$$

				The following proposition provides an essential property of a quasi-Fej\'{e}r monotone sequence; see Combettes \cite{C01} for further information.
				\begin{proposition}\label{fejer} \cite[Theorem 3.11]{C01} Let  $(x_k)_{k\geq1}$ be  a quasi-Fej\'{e}r monotone sequence relative to a nonempty subset $C\subset \R^n$. If at least one sequential cluster point of $(x_k)_{k\geq1}$ lies in $C$, then $(x_k)_{k\geq1}$ converges   to a point in $C$.
				\end{proposition}
			
				The following additional key tool  is known as the Silverman-Toeplitz theorem \cite{K04}.
				
				\begin{proposition}\label{Silverman-Toeplitz}
					Let  $(\alpha_k)_{k\geq1}$ be a positive real sequence with $\sum_{k=1}^{\infty}\alpha_k=+\infty$. If $(u_k)_{k\geq1}\subset\mathbb{R}^n$ is a sequence such that $\lim_{k\to+\infty}u_k=u\in\mathbb{R}^n$, then $\lim_{l\to+\infty}\frac{\sum_{k=1}^l\alpha_ku_k}{\sum_{k=1}^{l}\alpha_k}=u$.
				\end{proposition}
				
				We are  in a position to prove Theorem \ref{thm-main-2}, which is our main theorem.
				
				\begin{proof} Let $u\in\mathcal{S}$ and $k\geq1$ be fixed. For each $i=1,\ldots,m$, we have from the subdifferential inequality of $f_i$ that
					\begin{eqnarray*}
						\<\varphi_{i,k}-\varphi_{i+1,k}-\alpha_k\nabla h_i(\varphi_{i,k}),u-\varphi_{i+1,k}\>\leq \alpha_k(f_i(u)-f_i(\varphi_{i,k})).
					\end{eqnarray*}
					Moreover, by the convexity of $h_i$, we have	
					\begin{eqnarray}\label{thm-eqn1}
					&&2\<\varphi_{i,k}-\varphi_{i+1,k},u-\varphi_{i+1,k}\>\nonumber\\
					&\leq& 2\alpha_k(f_i(u)-f_i(\varphi_{i,k}))+2\alpha_k\<\nabla h_i(\varphi_{i,k}),u-\varphi_{i+1,k}\>\nonumber\\
					&=&2\alpha_k(f_i(u)-f_i(\varphi_{i,k}))+2\alpha_k\<\nabla h_i(\varphi_{i,k}),u-\varphi_{i,k}\>+2\alpha_k\<\nabla h_i(\varphi_{i,k}),\varphi_{i,k}-\varphi_{i+1,k}\>\nonumber\\
					&\leq&2\alpha_k(f_i(u)-f_i(\varphi_{i,k}))+2\alpha_k(h_i(u)-h_i(\varphi_{i,k}))+2\alpha_k\<\nabla h_i(\varphi_{i,k}),\varphi_{i,k}-\varphi_{i+1,k}\>\nonumber\\
					&\leq&2\alpha_k(F_i(u)-F_i(\varphi_{1,k}))+2\alpha_k(F_i(\varphi_{1,k})-F_i(\varphi_{i,k}))\nonumber\\
					&&+\alpha_k^2\|\nabla h_i(\varphi_{i,k})\|^2+\|\varphi_{i,k}-\varphi_{i+1,k}\|^2\nonumber\\
					&\leq&2\alpha_k(F_i(u)-F_i(\varphi_{1,k}))+2\alpha_k(F_i(\varphi_{1,k})-F_i(\varphi_{i,k}))\nonumber\\
					&&+2\alpha_k^2\|\nabla h_i(\varphi_{i,k})-\nabla h_i(u)\|^2+2\alpha_k^2\|\nabla h_i(u)\|^2+\|\varphi_{i,k}-\varphi_{i+1,k}\|^2.
					\end{eqnarray}
					Note that 
					\begin{eqnarray*}\|\varphi_{i+1,k}-u\|^2-\|\varphi_{i,k}-u\|^2+\|\varphi_{i+1,k}-\varphi_{i,k}\|^2=2\<\varphi_{i,k}-\varphi_{i+1,k},u-\varphi_{i+1,k}\>.
					\end{eqnarray*}
					Combining this equality with (\ref{thm-eqn1}), we obtain
					\begin{eqnarray*}\|\varphi_{i+1,k}-u\|^2-\|\varphi_{i,k}-u\|^2
						&\leq&2\alpha_k(F_i(u)-F_i(\varphi_{1,k}))\\
						&&+2\alpha_k(F_i(\varphi_{1,k})-F_i(\varphi_{i,k}))\\
						&&+2\alpha_k^2\|\nabla h_i(\varphi_{i,k})-\nabla h_i(u)\|^2+2\alpha_k^2\|\nabla h_i(u)\|^2.
					\end{eqnarray*}
					Summing up this inequality for all $i=1,\ldots,m$, we have 
					\begin{eqnarray*}\|\varphi_{m+1,k}-u\|^2-\|\varphi_{1,k}-u\|^2
						&\leq&2\alpha_k(F(u)-F(\varphi_{1,k}))\\
						&&+2\alpha_k\left(F(\varphi_{1,k})-\sum_{i=1}^mF_i(\varphi_{i,k})\right)\\
						&&+2\alpha_k^2\sum_{i=1}^m\|\nabla h_i(\varphi_{i,k})-\nabla h_i(u)\|^2+2\alpha_k^2\sum_{i=1}^m\|\nabla h_i(u)\|^2.
					\end{eqnarray*} 
					
					Since $(\varphi_{i,k})_{k\geq1}$ is bounded for all $i=1,\ldots,m$, there exists $M>0$ such that 
					\begin{eqnarray}\label{thm-eqn-M}\max\left\{\sum_{i=1}^m\|\nabla h_i(\varphi_{i,k})-\nabla h_i(u)\|^2,\max_{1\leq i\leq m}\|\nabla h_i(\varphi_{i,k})\|\right\}\leq M,
					\end{eqnarray}
					and so
					\begin{eqnarray}\label{thm-eqn2}\|\varphi_{m+1,k}-u\|^2-\|\varphi_{1,k}-u\|^2
					&\leq&2\alpha_k(F(u)-F(\varphi_{1,k}))\nonumber\\
					&&+2\alpha_k\left(F(\varphi_{1,k})-\sum_{i=1}^mF_i(\varphi_{i,k})\right)\nonumber\\
					&&+2\alpha_k^2M+2\alpha_k^2\sum_{i=1}^m\|\nabla h_i(u)\|^2.
					\end{eqnarray}
					
					On the other hand, since $\partial f_i$ maps a bounded subset into a bounded nonempty subset of $\R^n$ (see \cite[Proposition 16.20 (iii)]{BC11}), we have from the definition of $\partial f_i$ that for each $i=1,\ldots,m$, there exists $K_i>0$ such that
					\begin{eqnarray*}\|\varphi_{i,k}-\varphi_{i+1,k}-\alpha_k\nabla h_i(\varphi_{i,k})\|\leq\alpha_kK_i,
					\end{eqnarray*}
					and so
					\begin{eqnarray*}\|\varphi_{i,k}-\varphi_{i+1,k}\|\leq\alpha_kK_i+\alpha_k\|\nabla h_i(\varphi_{i,k})\|\leq\alpha_kK,
					\end{eqnarray*}
					where 
					\begin{eqnarray*}\label{thm-eqn2-K}K:=M+\max_{1\leq i \leq m}K_i.
					\end{eqnarray*}
					
					Note that
					\begin{eqnarray*}\|\varphi_{i,k}-\varphi_{1,k}\|\leq\sum_{j=1}^{i-1}\|\varphi_{j,k}-\varphi_{j+1,k}\|\leq\alpha_kiK.
					\end{eqnarray*}
					Moreover, since $\bigcup_{i=1}^m\{\varphi_{i,k}:k\geq 1\}$ is  bounded, by using \cite[Proposition 16.20 (ii)]{BC11} again, we know that the functions $f_i$ and $h_i$ are Lipschitz continuous on all bounded sets.  For all $i=1,\ldots,m$, there exists the Lipschitz constant $c_i>0$ such that
					\begin{eqnarray*}F_i(\varphi_{1,k})-F_i(\varphi_{i,k})\leq c_i\|\varphi_{1,k}-\varphi_{i,k}\|\leq c\|\varphi_{1,k}-\varphi_{i,k}\|\leq\alpha_kicK,
					\end{eqnarray*}
					where $c:=\max_{1\leq i\leq m}c_i$. This yields
					\begin{eqnarray*}F(\varphi_{1,k})-\sum_{i=1}^mF_i(\varphi_{i,k})\leq \alpha_kcK\sum_{i=1}^mi=\alpha_kcK\frac{m(m+1)}{2}.
					\end{eqnarray*}
					
					Combining this relation with (\ref{thm-eqn2}), we obtain
					\begin{eqnarray*}\|\varphi_{m+1,k}-u\|^2-\|\varphi_{1,k}-u\|^2
						&\leq&2\alpha_k(F(u)-F(\varphi_{1,k}))+\alpha_k^2cKm(m+1)\\
						&&+2\alpha_k^2M+2\alpha_k^2\sum_{i=1}^m\|\nabla h_i(u)\|^2.
					\end{eqnarray*}
					
					On the other hand, by the definition of $\varphi_{1,k}$, we note that
					\begin{eqnarray}\label{lemma-eqn-12-4-1}
					\|\varphi_{1,k}-u\|^2&=&\|\varphi_{1,k}-x_k\|^2+\|x_k-u\|^2+2\<\varphi_{1,k}-x_k,x_k-u\>\nonumber\\
					&=&\alpha_k^2\beta_k^2\|\nabla g(x_k)\|^2+\|x_k-u\|^2-2\alpha_k\beta_k\<\nabla g(x_k),x_k-u\>.
					\end{eqnarray}
					
					 Thanks to the Baillon-Haddad theorem \cite[Corollary 18.16]{BC11}, we know that $\nabla g$ is $\frac{1}{L_g}$-cocoercive. By using $\nabla g(u)=0$, we have
					\begin{eqnarray*}
						2\alpha_k\beta_k\<\nabla g(x_k),x_k-u\>
						&=&2\alpha_k\beta_k\<\nabla g(x_k)-\nabla g(u),x_k-u\>\nonumber\\
						&\geq&\frac{2\alpha_k\beta_k}{L_g}\|\nabla g(x_k)-\nabla g(u)\|^2=\frac{2\alpha_k\beta_k}{L_g}\|\nabla g(x_k)\|^2,
					\end{eqnarray*}
					which implies that
					\begin{eqnarray}\label{lemma-eqn-12-5-1}
					-2\alpha_k\beta_k\<\nabla g(x_k),x_k-u\>
					&\leq&-\frac{2\alpha_k\beta_k}{L_g}\|\nabla g(x_k)\|^2.
					\end{eqnarray}
					
					Thus, the inequalities (\ref{lemma-eqn-12-4-1}) and (\ref{lemma-eqn-12-5-1}) imply that
					\begin{eqnarray*}\|\varphi_{1,k}-u\|^2\leq\|x_k-u\|^2+\alpha_k\beta_k\left(\alpha_k\beta_k-\frac{2}{L_g}\right)\|\nabla g(x_k)\|^2.
					\end{eqnarray*}
					
					Using the last two inequalities and the assumption (H3), it follows that there exists $k_0\in\N$ such that 
					\begin{eqnarray*}\|x_{k+1}-u\|^2-\|x_k-u\|^2
						&\leq&2\alpha_k(F(u)-F(\varphi_{1,k}))+\alpha_k^2cKm(m+1)\\
						&&+2\alpha_k^2M+2\alpha_k^2\sum_{i=1}^m\|\nabla h_i(u)\|^2, \indent\forall k\geq k_0.
					\end{eqnarray*}

					Now, since $(x_k)_{k\geq1}$ is bounded, we let $z\in\R^n$ be its sequential cluster point and 	a subsequence $(x_{k_j})_{j\geq1}$ of $(x_k)_{k\geq1}$ be such that $x_{k_j}\to z$. By Lemma \ref{key-lemma4-2} (iii) and (iv) (see, Appendix A), we have $\varphi_{1,k_j}\to z$ and $z\in\argmin g$. Thus,  for every $k_j\geq k_0$, we have
					\begin{eqnarray*}
						2\alpha_{k_j}(F(\varphi_{1,k_j})-F(u))	&\leq&\|x_{k_j}-u\|^2-\|x_{{k_j}+1}-u\|^2\\
						&&+\alpha_{k_j}^2\left(cKm(m+1)+2M+2\sum_{i=1}^m\|\nabla h_i(u)\|^2\right),
					\end{eqnarray*}
					which yields
					\begin{eqnarray*}
						\sum_{k=k_0}^{k_j}\alpha_{k}(F(\varphi_{1,k})-F(u))	&\leq&\frac{\|x_{1}-u\|^2}{2}-\frac{\|x_{{k_j}+1}-u\|^2}{2}+M'\sum_{k=k_0}^{k_j}\alpha_{k}^2,
					\end{eqnarray*}
					where $M':=\frac{cKm(m+1)}{2}+M+\sum_{i=1}^m\|\nabla h_i(u)\|^2$.	Hence, we have
					\begin{eqnarray*}
						\frac{\sum_{k=k_0}^{k_j}\alpha_{k}(F(\varphi_{1,k})-F(u))}{\sum_{k=k_0}^{k_j}\alpha_{k}}	&\leq&\frac{\|x_{1}-u\|^2}{2\sum_{k=k_0}^{k_j}\alpha_{k}}+M'\frac{\sum_{k=k_0}^{k_j}\alpha_{k}^2}{\sum_{k=k_0}^{k_j}\alpha_{k}}.
					\end{eqnarray*}
					
					Consequently, by the assumption (H2) and Proposition \ref{Silverman-Toeplitz}, we obtain
					\begin{eqnarray*}
						\liminf_{j\to+\infty}	\frac{\sum_{k=k_0}^{k_j}\alpha_{k}(F(\varphi_{1,k})-F(u))}{\sum_{k=k_0}^{k_j}\alpha_{k}}	\leq0.
					\end{eqnarray*}
					The convexity of $F$ together with Proposition \ref{Silverman-Toeplitz} also yields 
					\begin{eqnarray*}
						F(z)\leq\liminf_{j\to+\infty}	F\left(\frac{\sum_{k=k_0}^{k_j}\alpha_{k}\varphi_{1,k}}{\sum_{k=k_0}^{k_j}\alpha_{k}}\right)\leq F(u).
					\end{eqnarray*}
					Since $u\in\mathcal{S}$ is arbitrary, we have $z\in\mathcal{S}$. Therefore, by Proposition \ref{fejer},  the sequence $(x_k)_{k\geq1}$ converges to a point in $\mathcal{S}$. 
				\end{proof}
				
					Some remarks relating to Theorem \ref{thm-main-2} are as follows.
					
						\begin{remark}\label{key-remark}
				\begin{itemize}
					\item[(i)] One can obtain a  convergence result as Theorem \ref{thm-main-2} in a general setting of  a proper convex lower semicontinuous objective  function $f_i:\R^n\to(-\infty,+\infty]$, provided that the Lipschitz continuity relative to all bounded subsets of the functions $f_i$ and $h_i$ and the fact that the subdifferential of $f_i$ maps bounded subsets of $\R^n$ into bounded nonempty subsets of $\R^n$ are imposed. The  properties that the functions $f_i, h_i$  are Lipschitz continuous on all bounded subsets of $\R^n$ is typically assumed in order to guarantee the non-ergodic convergence of incremental proximal type schemes; see, for instance, \cite{B11,B12}. In fact, there are several loss functions in machine learning which satisfy the Lipschitz continuous property, for instance, the hinge, logistic, and Huber  loss functions; see \cite{RDCPV04} for further discussion. 
					\item[(ii)] One can also obtain a weak ergodic convergence of the sequence $(x_k)_{k\geq1}$ in a general setting of real Hilbert space by slightly modifying the proofs of Theorem 3.1 and Corollary 4.1 in \cite{NP18-2}.
				\end{itemize}
			\end{remark}


		\section{Numerical Examples}
		
			In this section, we demonstrate the effectiveness of the proposed  algorithm by applying  to solve the image reconstruction problem addressing the image inpainting and the generalized Heron problems. All the experiments were performed under MATLAB 9.6 (R2019a) running on a MacBook Pro 13-inch, 2019 with a 2.4 GHz Intel Core i5 processor and 8 GB 2133 MHz LPDDR3 memory. 

	\subsection{Image Inpainting}
Let $n:=\ell_1\times\ell_2$ and $X\in\R^{\ell_1\times \ell_2}$ be an ideal complete image. Let  $x\in\R^n$ represent the  vector generated by vectorizing the  image $X$. Let $\mathbf{b}\in\R^n$ be the marked image and $\mathbf{B}\in\R^{n\times n}$ be the diagonal matrix where  $\mathbf{B}_{i,i}=0$ if the pixel $i$ in the marked image $\mathbf{b}$ is missing (in our experiments, we set it to be black) and $\mathbf{B}_{i,i}=1$ otherwise, for $i=1,\ldots,n$. 
In a traditional way, the image inpainting problem aims to reconstruct the clean image $x$ from the marked image $\mathbf{b}$ by solving   the unconstrained nonsmooth optimization problem 	\begin{eqnarray}\label{inpaint-pb-trad}%
\begin{array}{ll}
\textrm{minimize }\indent  \lambda_1\|Wx\|_1+\frac{\lambda_2}{2}\|x\|^2+\frac{1}{2}\|\mathbf{B}\cdot-\mathbf{b}\|^2\\
\textrm{subject to}\indent x\in\R^n,\\
\end{array}%
\end{eqnarray} 
where $\lambda_1,\lambda_2>0$ are the penalization parameters, and $W$ is the inverse discrete Haar wavelet transform. The term $\|Wx\|_1$ is to deduce the sparsity of the image under the wavelet transformation, and the term $\frac{1}{2}\|x\|^2$ is to deduce the uniqueness of the solution. Note that   $W^\top W=WW^\top=I$ and so $\|W^\top W\|=1$. For more details of wavelet-based inpainting, see \cite{SMF10}. Meanwhile, in our experiment, we consider the basic structure of the ill-conditional linear inverse problem $\mathbf{B}x=\mathbf{b}$ and make use of the regularized tools but under the framework of the  nonsmooth convex constrained minimization problem
		\begin{eqnarray}\label{inpaint-pb}%
			\begin{array}{ll}
				\textrm{minimize }\indent  \lambda_1\|Wx\|_1+\frac{\lambda_2}{2}\|x\|^2\\
				\textrm{subject to}\indent x\in\argmin \frac{1}{2}\|\mathbf{B}\cdot-\mathbf{b}\|^2.\\
			\end{array}%
		\end{eqnarray} 				
	Note that the problem (\ref{inpaint-pb}) fits into the setting of the problem (\ref{MP}) where  $m=1$, 	$f_1=\lambda_1\|W(\cdot)\|_1$, $h_1=\frac{\lambda_2}{2}\|\cdot\|^2$, and $g=\frac{1}{2}\|\mathbf{B}\cdot-\mathbf{b}\|^2$.
		To show the performance of the proposed method, we solve the image inpainting problem (\ref{inpaint-pb}) using Algorithm \ref{algorithm-ergodic-smooth}. 
Firstly, we test the method by presenting the ISNR values for different combinations of parameters  for reconstructing  the 384$\times$512 peppers image whose noisy image is obtained by  randomly masking 60\% of all pixels to black. The quality of the reconstructed images is measured by means of the improvement in signal-to-noise ratio (ISNR) in decibel (dB), that is,
		$$\mathrm{ISNR}(k)=10\log_{10}\left( \frac{\|x-\mathbf{b}\|^2}{\|x-x_k\|^2}\right),$$
		where $x, \mathbf{b}$, and $x_k$ denote the original clean image, the noisy image with missing pixels, and the reconstructed image at iteration $k$, respectively.    We run the algorithm for 20 iterations to obtain the ISNR value whose results are presented in Tables \ref{tb-isnr} and \ref{tb2-isnr}.
		\begin{table}[H]
			\centering
			\setlength{\tabcolsep}{6pt}
			{\scriptsize		\caption{\label{tb-isnr} ISNR values  after 20 iterations for different choices of penalization parameters $\lambda_1$ and $\lambda_2$  with the step size  $\alpha_k=1/k$ and the penalization parameter $\beta_k=k$.}
				\begin{tabular}{@{}l   r r  r r  r r  r r  r r  r r r r r r r r r@{}}
					\toprule  $\lambda_1$ $\rightarrow$  & \multirow{2}{*}{$0.1$} & \multirow{2}{*}{$0.2$} & \multirow{2}{*}{$0.3$} & \multirow{2}{*}{$0.4$} &  \multirow{2}{*}{$0.5$} & \multirow{2}{*}{$0.6$}  & \multirow{2}{*}{$0.7$}  & \multirow{2}{*}{$0.8$} & \multirow{2}{*}{$0.9$} & \multirow{2}{*}{$1$} & \multirow{2}{*}{$1.5$} &  \multirow{2}{*}{$2$} \\
					
					 $\lambda_2$  $\downarrow$  	  	 \\ \midrule
$10^{-8}$	&3.1701	&5.9704	&8.6229	&11.1305	&13.2470	&14.6711	&15.4206	&15.7265	&15.8401	&15.8574	&15.4991	&14.9656\\
$10^{-5}$	&3.1700	&5.9703	&8.6227	&11.1304	&13.2468	&14.6709	&15.4205	&15.7264	&15.8400	&15.8573	&15.4990	&14.9656\\
$10^{-4}$	&3.1696	&5.9695	&8.6215	&11.1288	&13.2452	&14.6696	&15.4195	&15.7258	&15.8395	&15.8569	&15.4987	&14.9653\\
$10^{-3}$	&3.1654	&5.9616	&8.6096	&11.1135	&13.2288	&14.6561	&15.4100	&15.7189	&15.8341	&15.8523	&15.4957	&14.9625\\
0.005	&3.1466   &5.9267	&8.5569	&11.0458	&13.1562	&14.5957	&15.3671	&15.6880	&15.8097	&15.8319	&15.4820	&14.9502\\
0.01	       &3.1233	&5.8835	&8.4916	&10.9615	&13.0654	&14.5192	&15.3122	&15.6484	&15.7786	&15.8058	&15.4646	&14.9346\\
0.05	       &2.9436	&5.5502	&7.9894	&10.3072	&12.3390	&13.8716	&14.8185	&15.2887	&15.4972	&15.5707	&15.3129	&14.8020\\
0.1	       &2.7353	&5.1635	&7.4109	&9.5430	&11.4549	&13.0110	&14.0903	&14.7313	&15.0632	&15.2124	&15.0952	&14.6188					\\ \bottomrule
				\end{tabular}
			}
		\end{table}
		
		In Table \ref{tb-isnr}, we list the values of the ISNR  after 20 iterations are performed by Algorithm \ref{algorithm-ergodic-smooth} for different choices of the penalization parameters $\lambda_1,\lambda_2> 0$. In this case, we put the step size sequence $\alpha_k=1/k$ and the penalization sequence $\beta_k=k$. We observe that  combining $\lambda_1=1$ with each parameter $\lambda_2\in(0,10^{-4}]$ leads to a large ISNR value of approximately 15.86 dB.	Moreover, one can see that  the ISNR value tends to increase when $\lambda_1\in[0.1,1]$ increases, while it tends to decrease when $\lambda_2$ increases.
		
			\begin{table}[H]
			\centering
			\setlength{\tabcolsep}{6pt}
			{\tiny	\caption{\label{tb2-isnr} ISNR values    after 20 iterations for different choices of  step size  $\alpha_k=a/k$ and penalization parameter $\beta_k=bk$ when $\lambda_1=1$ and $\lambda_2=10^{-4}$.}
				\begin{tabular}{@{}l   r r  r r  r r  r r  r r  r r r r r r r r r r r r r r r r r r@{}}
					\toprule	$a$ $\rightarrow$ & \multirow{2}{*}{$0.8$} & \multirow{2}{*}{$0.9$} &  \multirow{2}{*}{$1$} & \multirow{2}{*}{$1.1$}  & \multirow{2}{*}{$1.2$}  & \multirow{2}{*}{$1.3$} &  \multirow{2}{*}{$1.4$} & \multirow{2}{*}{$1.5$}  & \multirow{2}{*}{$1.6$} & \multirow{2}{*}{$1.7$} &  \multirow{2}{*}{$1.8$} & \multirow{2}{*}{$1.9$} & \multirow{2}{*}{$2$} \\ 
							 $b$  $\downarrow$  	  	 \\ \midrule
0.5		&12.0819	&12.9435	&13.5459	&13.9563	&14.2388	&14.4386	&14.5835	&14.6916	&14.7741	&14.8385	&14.8897	&14.9313	&14.9656\\
0.6		&13.2543	&13.9796	&14.4414	&14.7384	&14.9360	&15.0727	&15.1705	&15.2430	&15.2981	&15.3409	&15.3748	&15.4022	&15.4245\\
0.7		&14.0487	&14.6485	&15.0080	&15.2330	&15.3802	&15.4807	&15.5519	&15.6042	&15.6438	&15.6744	&15.6984	&15.7172	&15.7323\\
0.8		&14.6041	&15.1009	&15.3895	&15.5675	&15.6822	&15.7594	&15.8139	&15.8537	&15.8832	&15.9056	&15.9228	&15.9361	&15.9465\\
0.9		&14.9997	&15.4189	&15.6594	&15.8054	&15.8984	&15.9597	&16.0030	&16.0341	&16.0569	&16.0737	&16.0861	&16.0951	&16.1017\\
1		&15.2883	&15.6511	&15.8574	&15.9807	&16.0579	&16.1084	&16.1437	&16.1686	&16.1861	&16.1986	&16.2074	&16.2136	&-\\
1.1		&15.5002	&15.8237	&16.0060	&16.1133	&16.1792	&16.2221	&16.2516	&16.2719	&16.2860	&16.2956	&16.3019	&-	&-\\
1.2		&15.6600	&15.9543	&16.1182	&16.2135	&16.2715	&16.3090	&16.3345	&16.3518	&16.3636	&-	&-	&-	&-\\
1.3		&15.7859	&16.0569	&16.2063	&16.2922	&16.3439	&16.3775	&16.3998	&16.4147	&-	&-	&-	&-	&-\\
1.4	&15.8859	&16.1386	&16.2766	&16.3551	&16.4019	&16.4320	&16.4519	&-	&-	&-	&-	&-	&-\\
1.5		&15.9625	&16.2011	&16.3305	&16.4032	&16.4467	&16.4749	&-	&-	&-	&-	&-	&-	&-\\
1.6		&16.0223	&16.2499	&16.3724	&16.4406	&16.4819	&-	&-	&-	&-	&-	&-	&-	&-\\
1.7		&16.0697	&16.2885	&16.4062	&16.4713	&-	&-	&-	&-	&-	&-	&-	&-	&-\\
1.8		&16.1088	&16.3210	&16.4349	&16.4981	&-	&-	&-	&-	&-	&-	&-	&-	&-\\
1.9	&16.1418	&16.3481	&16.4595	&-	&-	&-	&-	&-	&-	&-	&-	&-	&-\\
2	&16.1690	&16.3710	&-	&-	&-	&-	&-	&-	&-	&-	&-	&-	&-		\\ \bottomrule
				\end{tabular}
			}
		\end{table}
		
		In Table \ref{tb2-isnr}, we present the values of the ISNR  after 20 iterations are performed by  Algorithm \ref{algorithm-ergodic-smooth} for different choices of   the positive square summable step size sequence $\alpha_k=a/k$  where $a\in[0.8,2]$, and  the positive penalization sequence $\beta_k=bk$, where $b\in[0.5,2]$. Note that the results for the combinations dissatisfying the condition (H3) are not presented in the table. 
		Observe that combining $\alpha_k=1.1/k$ with $\beta_k=1.8k$ leads to the largest ISNR value of 16.4981 dB. Furthermore, we notice that  the ISNR value tends to increase when  both $a$ and $b$ increase.   
		
		Next, we will compare the performance of  Algorithm \ref{algorithm-ergodic-smooth} when solving the inpainting problem in our setting (\ref{MP}) via the ISNR values with other well-known iterative methods, namely, proximal-gradient method (PGM) (see \cite[Theorem 25.8]{BC11})
		  and the fast iterative shrinkage-thresholding algorithm (FISTA) \cite{BT09}. These methods are suited for solving the classical setting (\ref{inpaint-pb-trad}) when putting $f=\lambda_1\|W(\cdot)\|_1$ and  $h=\frac{\lambda_2}{2}\|\cdot\|^2+\frac{1}{2}\|\mathbf{B}\cdot-\mathbf{b}\|^2$ with $\nabla h$ being $(\lambda_2+1)-$Lipschitz continuous. For fair comparison, we manually choose the best possible parameters combinations of each method (see Tables \ref{tb-isnr-fista} - \ref{tb-isnr-pgm-2} in Appendix B for several parameters combinations of PGM and FISTA) as follows:
		\begin{itemize}
			\item Algorithm \ref{algorithm-ergodic-smooth}: $\lambda_1=1, \lambda_2=10^{-4}, \alpha_k=1.1/k$, and $\beta_k=1.8k$;
			\item PGM \cite[Theorem 25.8]{BC11}:  $\lambda_1=0.1, \lambda_2=10^{-8}$, and $\gamma=1.9/(\lambda_2+1)$;
				\item FISTA \cite{BT09}:  $\lambda_1=0.05$ and  $\lambda_2=10^{-4}$.
			\end{itemize}
		We testify the methods by  the  relative change
		$$\max\left\{\frac{\|x_{x+1}-x_k\|}{\|x_k\|+1},\frac{|A(x_{k+1})-A(x_{k})|}{|A(x_{k})|+1},\frac{|B(x_{k+1})-B(x_{k})|}{|B(x_{k})|+1}\right\}\leq \epsilon,$$
		where  $A(\cdot):=\lambda_1\|W(\cdot)\|_1+\frac{\lambda_2}{2}\|\cdot\|^2$, $B(\cdot):=\frac{1}{2}\|\mathbf{B}\cdot-\mathbf{b}\|^2$, and $\epsilon$ is an optimality tolerance. We use the optimality tolerance $10^{-6}$ for obtaining the ISNR values. Moreover, we   show the curves of ISNR  for	the reconstructed images performed by these three methods after 50 iterations. We test the methods for three test images, and the results are shown in Figures \ref{PP} - \ref{LH}. 
		\begin{figure}[H]	
			\begin{minipage}[t]{0.22\textwidth}
				\centering
				\resizebox*{3.5cm}{!}{\includegraphics{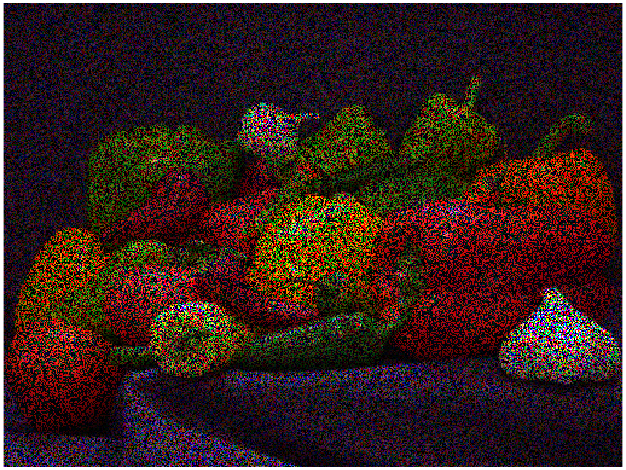}}\\				
				{\scriptsize (a) 60\% missing pixels \\$384\times 512$ peppers image}
			\end{minipage}	
			\begin{minipage}[t]{0.22\textwidth}
				\centering
				\resizebox*{3.5cm}{!}{\includegraphics{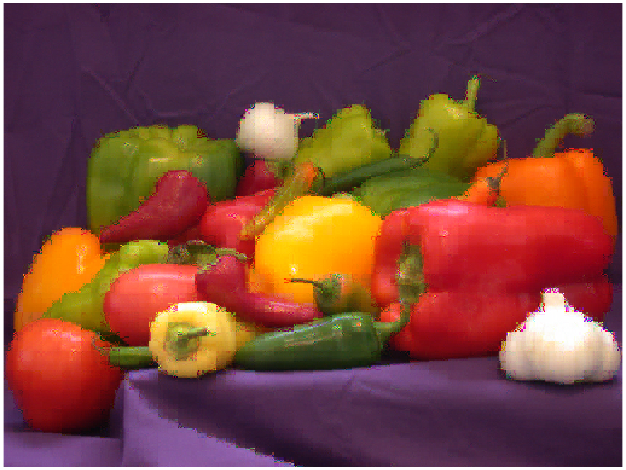}}\\				
				{\scriptsize	(b) Algorithm 1,  \\ISNR = 16.5561 dB}
			\end{minipage}			
				\begin{minipage}[t]{0.22\textwidth}
				\centering
				\resizebox*{3.5cm}{!}{\includegraphics{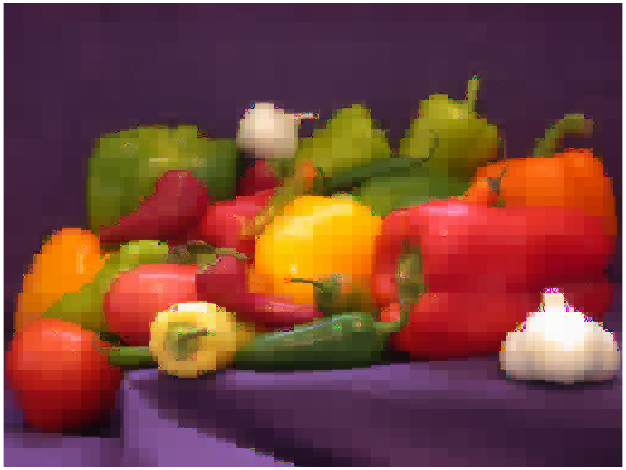}}\\			
				{\scriptsize	(c) FISTA, \\ISNR = 16.1556 dB}
			\end{minipage}			
				\begin{minipage}[t]{0.22\textwidth}
				\centering
				\resizebox*{3.5cm}{!}{\includegraphics{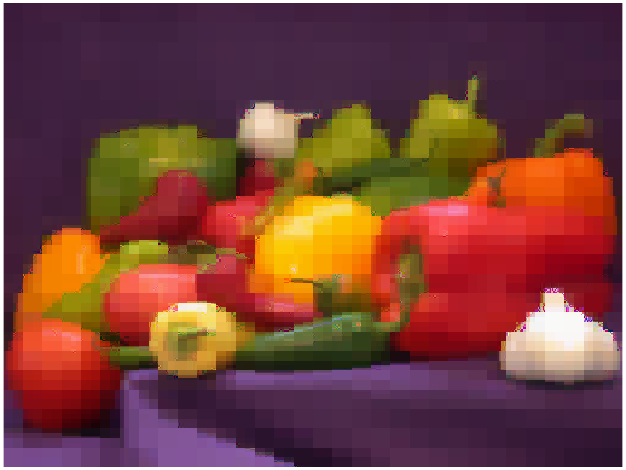}}\\				
				{\scriptsize	(d) PGM, \\ ISNR = 15.1692 dB}
			\end{minipage}			
	
		\vspace{0.3cm}
			\begin{minipage}[t]{0.4\textwidth}
				\centering
				\resizebox*{6cm}{!}{\includegraphics{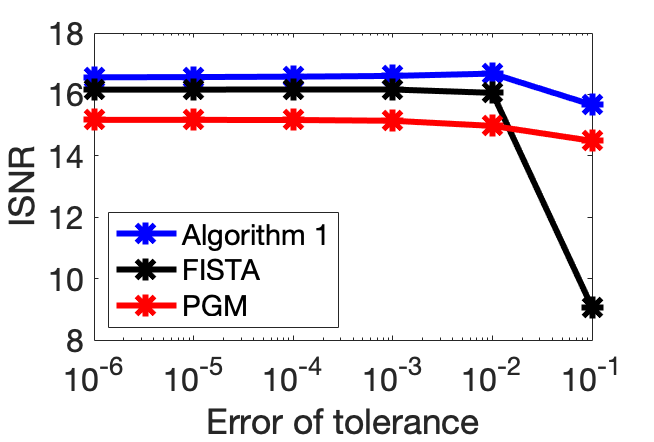}}\\				
				{\scriptsize	(e) ISNR values via error of tolerances}
					\end{minipage}					
					\begin{minipage}[t]{0.4\textwidth}
					\centering
					\resizebox*{6cm}{!}{\includegraphics{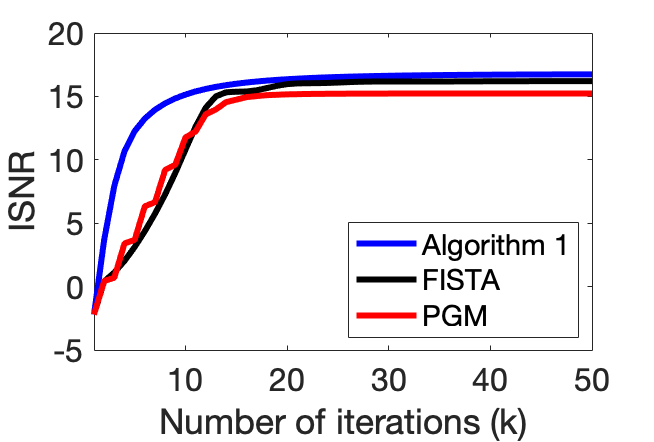}}\\
				
					{\scriptsize	(f) ISNR values via iterations}
					\end{minipage}	
		\caption{\label{PP} Image Inpainting. Figure (a) shows the  384$\times$512 peppers image with randomly masking 60\% of all pixels to black. Figures (b) - (d) show the reconstructed images  performed by Algorithm \ref{algorithm-ergodic-smooth}, FISTA, and PGM, respectively, for the optimality tolerance $10^{-6}$.  Figure (e) shows ISNR values when the iterates reach various optimality tolerances, and Figure (f) shows ISNR values when performing 50 iterations.}
\end{figure}

\begin{figure}[H]	
	\begin{minipage}[t]{0.2\textwidth}
		\centering
		\resizebox*{3cm}{!}{\includegraphics{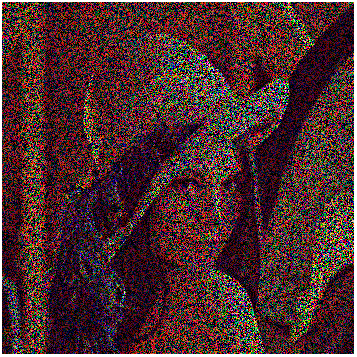}}\\				
		{\scriptsize (a) 60\% missing pixels \\$512\times 512$ Lena image}
	\end{minipage}	
	\begin{minipage}[t]{0.2\textwidth}
		\centering
		\resizebox*{3cm}{!}{\includegraphics{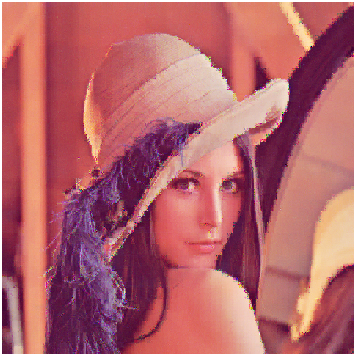}}\\				
		{\scriptsize	(b) Algorithm 1,  \\ISNR = 18.4371 dB}
	\end{minipage}			
	\begin{minipage}[t]{0.2\textwidth}
		\centering
		\resizebox*{3cm}{!}{\includegraphics{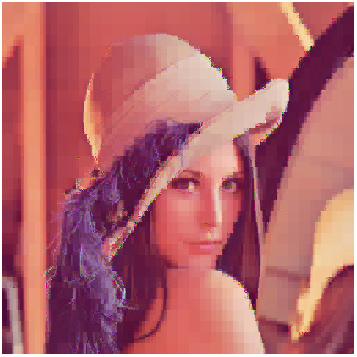}}\\			
		{\scriptsize	(c) FISTA, \\ISNR = 18.1711 dB}
	\end{minipage}			
	\begin{minipage}[t]{0.2\textwidth}
		\centering
		\resizebox*{3cm}{!}{\includegraphics{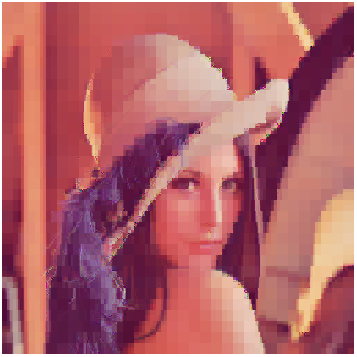}}\\				
		{\scriptsize	(d) PGM, \\ ISNR = 17.2952 dB}
	\end{minipage}			
	
	\vspace{0.3cm}
	\begin{minipage}[t]{0.4\textwidth}
		\centering
		\resizebox*{6cm}{!}{\includegraphics{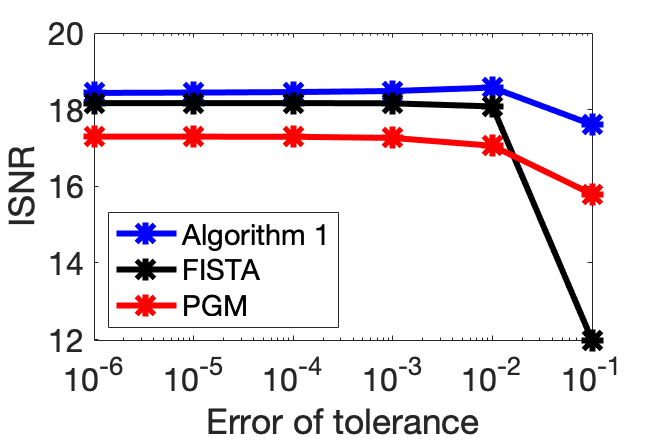}}\\				
		{\scriptsize	(e) ISNR values via error of tolerances}
	\end{minipage}					
	\begin{minipage}[t]{0.4\textwidth}
		\centering
		\resizebox*{6cm}{!}{\includegraphics{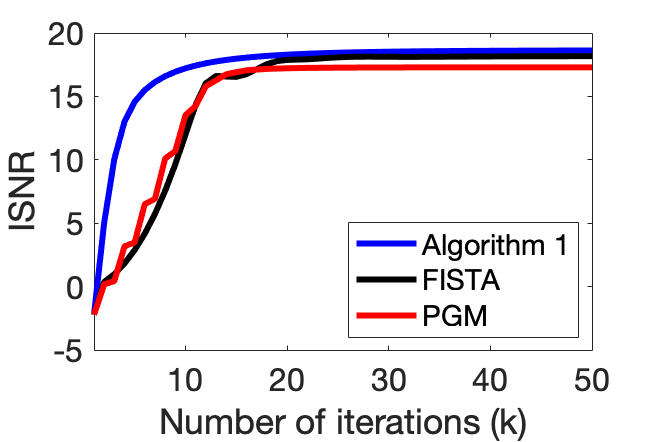}}\\
		
		{\scriptsize	(f) ISNR values via iterations}
	\end{minipage}	
		\caption{\label{LE}  Image Inpainting. Figure (a) shows the  512$\times$512 Lenna   image with randomly masking 60\% of all pixels to black. Figures (b) - (d) show the reconstructed images  performed by Algorithm \ref{algorithm-ergodic-smooth}, FISTA, and PGM, respectively, for the optimality tolerance $10^{-6}$.  Figure (e) shows ISNR values when the iterates reach various optimality tolerances, and Figure (f) shows ISNR values when performing 50 iterations.}
\end{figure}

\begin{figure}[H]	
	\begin{minipage}[t]{0.2\textwidth}
		\centering
		\resizebox*{3cm}{!}{\includegraphics{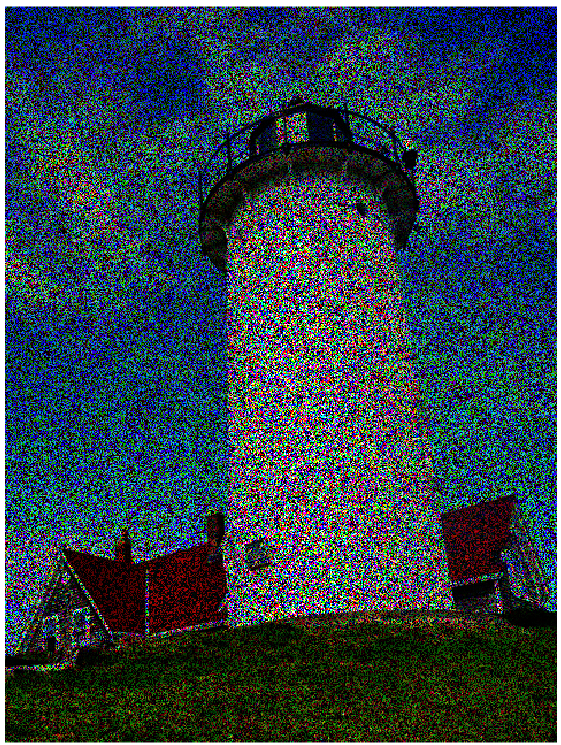}}\\				
		{\scriptsize (a) 60\% missing pixels \\$640\times 480$ lighthouse image}
	\end{minipage}	
	\begin{minipage}[t]{0.2\textwidth}
		\centering
		\resizebox*{3cm}{!}{\includegraphics{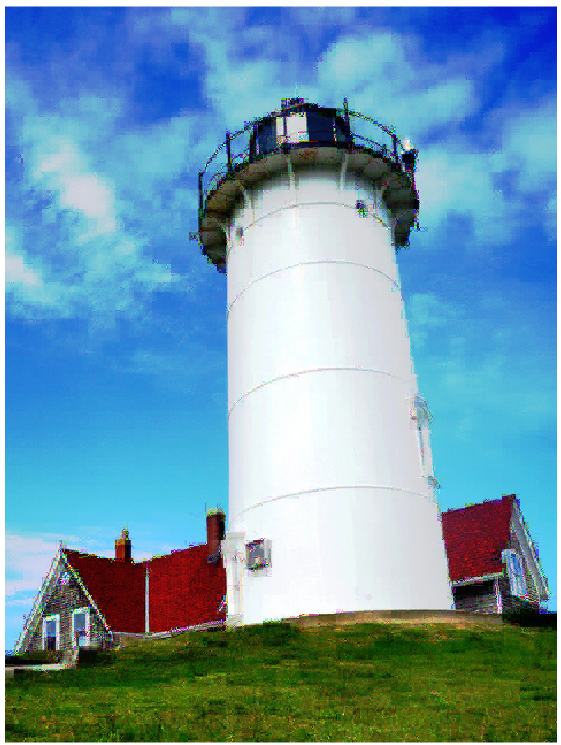}}\\				
		{\scriptsize	(b) Algorithm 1,  \\ISNR = 17.9481 dB}
	\end{minipage}			
	\begin{minipage}[t]{0.2\textwidth}
		\centering
		\resizebox*{3cm}{!}{\includegraphics{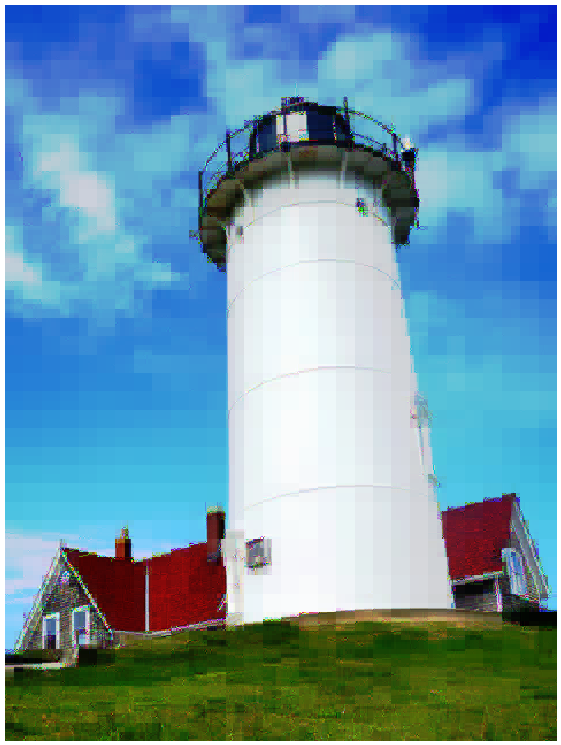}}\\			
		{\scriptsize	(c) FISTA, \\ISNR = 17.9159 dB}
	\end{minipage}			
	\begin{minipage}[t]{0.2\textwidth}
		\centering
		\resizebox*{3cm}{!}{\includegraphics{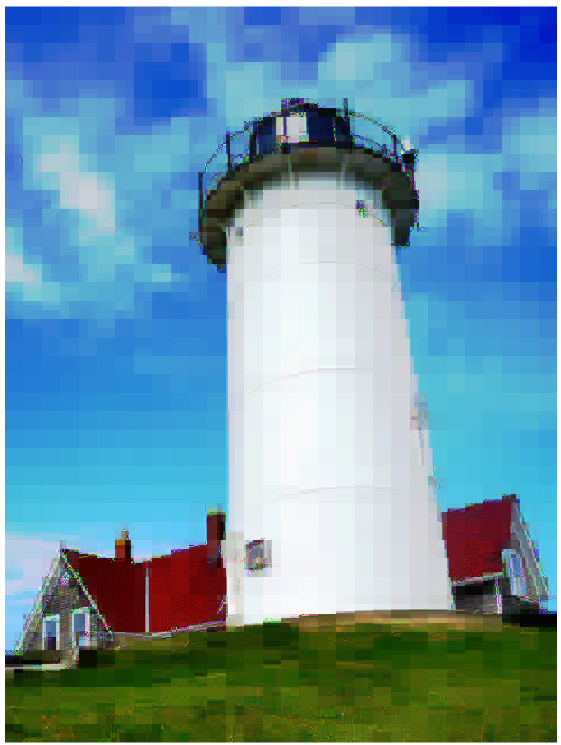}}\\				
		{\scriptsize	(d) PGM, \\ ISNR = 17.4031 dB}
	\end{minipage}			
	
	\vspace{0.3cm}
	\begin{minipage}[t]{0.4\textwidth}
		\centering
		\resizebox*{6cm}{!}{\includegraphics{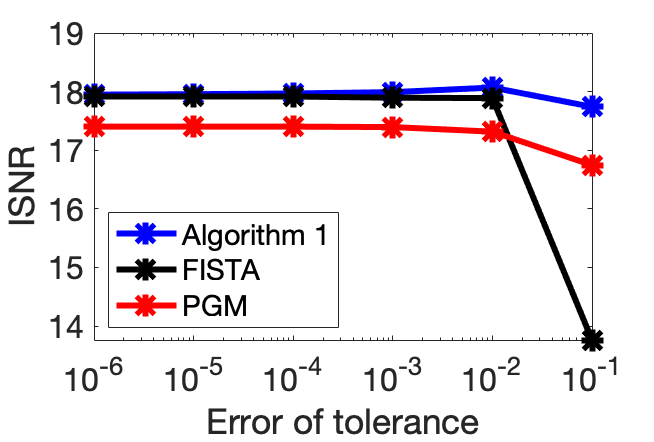}}\\				
		{\scriptsize	(e) ISNR values via error of tolerances}
	\end{minipage}					
	\begin{minipage}[t]{0.4\textwidth}
		\centering
		\resizebox*{6cm}{!}{\includegraphics{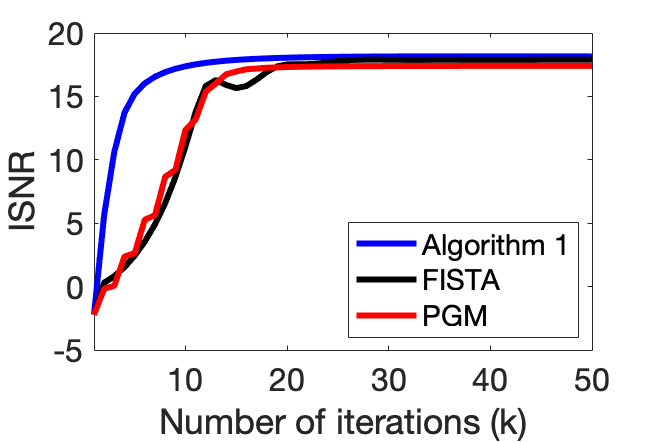}}\\
		
		{\scriptsize	(f) ISNR values via iterations}
	\end{minipage}	
			\caption{\label{LH} Image Inpainting. Figure (a) shows the 640$\times$480 lighthouse  image with randomly masking 60\% of all pixels to black. Figures (b) - (d) show the reconstructed images  performed by Algorithm \ref{algorithm-ergodic-smooth}, FISTA, and PGM, respectively, for the optimality tolerance $10^{-6}$.  Figure (e) shows ISNR values when the iterates reach various optimality tolerances, and Figure (f) shows ISNR values when performing 50 iterations.}
		\end{figure}

		 From all the above results, we observe that the proposed method (Algorithm \ref{algorithm-ergodic-smooth}) outperforms PGM and FISTA in the terms of the improvement in signal-to-noise ratio (ISNR) in both the error of tolerance criteria and stopping criteria with a fixed number of iterations, which may benefit from the usefulness of hierarchical setting (\ref{inpaint-pb}) considered in this work. 
	
        
		\subsection{Generalized Heron Problem}
		
		The traditional Heron problem is to find  a point on a straight line in a plane in which it minimizes the sum of distances from it to two given
points. Several generalizations of the classical Heron problem of finding a point that minimizes the sum of the distances to given closed convex sets over a nonempty simple closed convex set have been investigated by many authors, for instance, \cite{BCH15,MNS12,MNS12-2}.  However, it is very challenging to solve the generalized Heron problem when the constrained set is an affine subspace $\mathbf{A}x=\mathbf{b}$ (a solution set to a system of linear equations), which typically has no solution and thus computing a metric projection onto this feasible set is impossible. In addition, any methods mentioned in the above references cannot be applied in such case. This motivates us to consider the generalized Heron problem of finding a point that minimizes the sum of the distances to given closed convex sets over a least squares solution to a system of linear equation, that is,
		\begin{eqnarray}\label{heron}
			\begin{array}{ll}
				\textrm{minimize}\indent \sum_{i=1}^m\dist(x,C_i)+\frac{1}{2}\|x\|^2\\
				\textrm{subject to}\indent x\in\argmin \frac{1}{2}\|\mathbf{A}x-\mathbf{b}\|^2,
			\end{array}%
		\end{eqnarray}
		where $C_i \subset \mathbb{R}^n$ are nonempty closed convex subsets, for all $i=1,...,m$, $\mathbf{A}\in \mathbb{R}^{r\times n}$ is a matrix, and $\mathbf{b}\in \mathbb{R}^{r}$ is a vector. 
			We observe that (\ref{heron})   fits into the setting of the problem (\ref{MP}) when setting $f_i(x)=\dist(x,C_i)$,  $h_i(x)=\frac{1}{2m}\|x\|^2$, for all $i=1,...,m$, and $g(x)=\frac{1}{2}\|\mathbf{A}x-\mathbf{b}\|^2$ for all $x\in\R^n$.  In this case, a solution set of a system of linear equations can be empty. Moreover, it is worth noting that when performing our proposed method (Algorithm \ref{algorithm-ergodic-smooth}), we only require to compute the gradient of $g$ but not the inverse of any matrix. Note also that,  the square of $\ell_2$-norm is to guarantee the uniqueness of a solution to the problem.

		We will perform our experiments by considering the closed convex target sets $C_i \subset \mathbb{R}^n$, for all $i=1,\ldots,m$, which are balls of radius $0.2$ whose centers are created randomly in the interval $(-n^2,n^2)$. We put $r=m^2$ and generate all arrays of the matrix $\mathbf{A}$  randomly from the interval $(-n^2,n^2)$.  Our experiments will be divided into two cases, namely, the case of consistent constraint  where $\mathbf{b}=\mathbf{0}_{\mathbb{R}^{m^2}}$, and  the case of inconsistent constraint where $\mathbf{b}$ is not a zero vector.  In all experiments, we measure the performance of Algorithm \ref{algorithm-ergodic-smooth} by the relative change between two consecutive iterations, i.e.,
		$$\max\left\{\frac{\|x_{x+1}-x_k\|}{\|x_k\|+1},\frac{|F(x_k)-F(x_{k-1})|}{|F(x_{k-1})|+1},\frac{|g(x_k)-g(x_{k-1})|}{|g(x_{k-1})|+1}\right\}\leq \epsilon,$$
		where $F:=\sum_{i=1}^mF_i$ and $F_i:=f_i+h_i$ for all $i=1,...,m$. The runtimes of Algorithm \ref{algorithm-ergodic-smooth}  are clocked in seconds. Two  examples of an iteration $x_k$ generated by Algorithm \ref{algorithm-ergodic-smooth}  are illustrated in Figure \ref{cons}.

		\begin{figure}[H]
		\begin{minipage}[c]{0.5\textwidth}
			\centering
			\resizebox*{4.5cm}{!}{\includegraphics{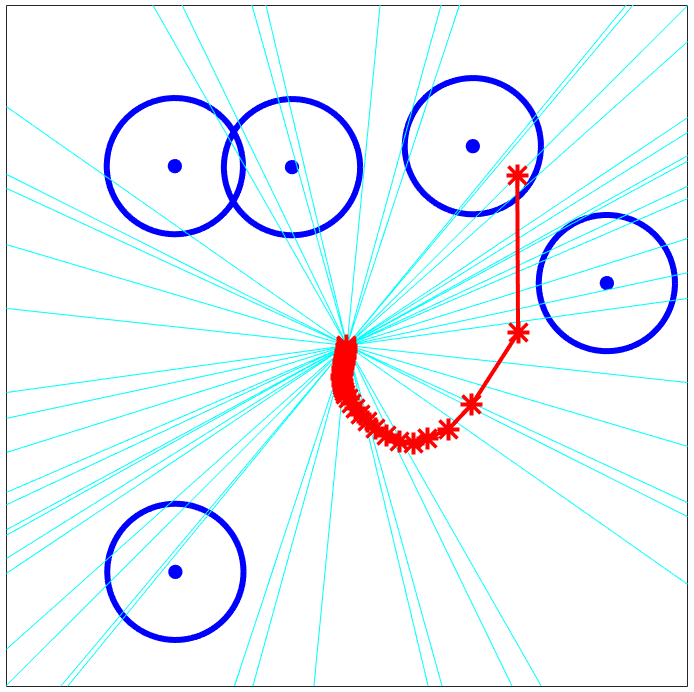}}	\\
			{\scriptsize	(a) consistent constraint}		
				\end{minipage}
    	\begin{minipage}[c]{0.4\textwidth}
		\centering
		\resizebox*{4.5cm}{!}{\includegraphics{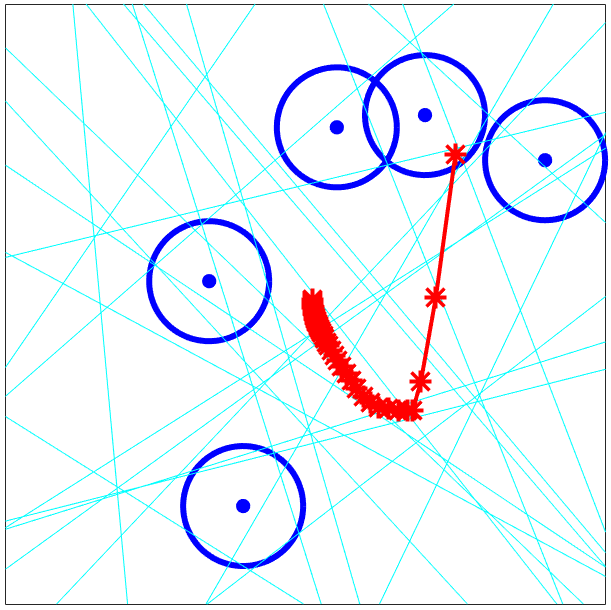}}	\\
	     	{\scriptsize	(b) inconsistent constraint}			     
        \end{minipage}
		\caption{\label{cons}Behavior of $x_k$ for several optimal tolerances $\epsilon$ on  the generalized Heron problems with consistent constraint (Figure (a)) and inconsistent constraint (Figure (b)).}
	\end{figure}

Next, we will consider the influences of corresponding parameters $\alpha_k$ and $\beta_k$ in Algorithm \ref{algorithm-ergodic-smooth} on  the generalized Heron problem with the consistent constraint, that is, $\mathbf{b}=\mathbf{0}_{\mathbb{R}^{m^2}}$. We use the number of target sets $m=5$, the dimension $n=2$, and perform $10$ samplings for the different randomly chosen matrix $\mathbf{A}\in \mathbb{R}^{25\times 5}$, the balls $C_1,\ldots,C_5\subset\mathbb{R}^2$, and the starting point $x_1\in\mathbb{R}^2$. We terminate Algorithm \ref{algorithm-ergodic-smooth} when the relative changes are less than $10^{-6}$, and then compute the averages of runtimes,  whose results are  shown in Table \ref{tb-cons}.	
		\begin{table}[H]
		\centering
		\setlength{\tabcolsep}{6pt}
		{\scriptsize		\caption{\label{tb-cons} Comparison of  runtime   for  different choices of step size  $\alpha_k = a/k$ and penalization parameter $\beta_k=bk/\|\mathbf{A}\|^2$ in  the case of consistent constraint.}
			\begin{tabular}{@{}l   r r  r r  r r  r r  r r  r r r r r r r r r@{}}
				\toprule  $a$ $\rightarrow$  & \multirow{2}{*}{$0.1$} & \multirow{2}{*}{$0.2$} & \multirow{2}{*}{$0.3$} & \multirow{2}{*}{$0.4$} &  \multirow{2}{*}{$0.5$} & \multirow{2}{*}{$0.6$}  & \multirow{2}{*}{$0.7$}  & \multirow{2}{*}{$0.8$} & \multirow{2}{*}{$0.9$} & \multirow{2}{*}{$1$} \\
				
				$b$  $\downarrow$  	  	 \\ \midrule
		0.1&	0.1152&	0.0988&	0.0977&	0.0800&	0.0967&	0.0732&	0.0848&	0.0756&	0.0830&	0.0813\\
		0.2&	0.0637&	0.0663&	0.0634&	0.0545&	0.0562&	0.0648&	0.0552&	0.0505&	0.0575&	0.0633\\
		0.3&    0.0524&	0.0465&	0.0461&	0.0483&	0.0609&	0.0528&	0.0427&	0.0481&	0.0485&	0.0542\\
		0.4&	0.0503&	0.0430&	0.0458&	0.0485&	0.0454&	0.0500&	0.0454&	0.0487&	0.0361&	0.0392\\
		0.5&	0.0413&	0.0393&	0.0441&	0.0409&	0.0436&	0.0383&	0.0371&	0.0385&	0.0354&	0.0437\\
		0.6&	0.0317&	0.0360&	0.0409&	0.0386&	0.0342&	0.0403&	0.0379&	0.0358&	0.0332&	0.0301\\
		0.7&	0.0337&	0.0330&	0.0338&	0.0352&	0.0345&	0.0280&	0.0333&	0.0330&	0.0366&	0.0358\\
		0.8&	0.0307&	0.0290&	0.0285&	0.0296&	0.0352&	0.0327&	0.0350&	0.0332&	0.0301&	0.0296\\
		0.9&	0.0347&	0.0308&	0.0293&	0.0330&	0.0324&	0.0364&	0.0297&	0.0300&	0.0297&	0.0297\\
		1.0&	0.0293&	0.0283&	0.0322&	0.0278&	0.0315&	0.0306&	0.0337&	0.0292&	0.0283&	0.0337\\
		1.1&	0.0298&	0.0254&	0.0301&	0.0290&	0.0354&	0.0314&	0.0270&	0.0264&	0.0252&	0.0287\\
		1.2&	0.0284&	0.0301&	0.0296&	0.0263&	0.0276&	0.0259&	0.0299&	0.0223&	0.0281&	0.0303\\
		1.3&	0.0238&	0.0267&	0.0259&	0.0286&	0.0292&	0.0244&	0.0263&	0.0260&	0.0268&	0.0302\\
		1.4&	0.0270&	0.0261&	0.0246&	0.0248&	0.0404&	0.0220&	0.0292&	0.0280&	0.0244&	0.0291\\
		1.5&	0.0252&	0.0279&	0.0261&	0.0233&	0.0329&	0.0244&	0.0253&	0.0270&	0.0268&	0.0280\\
		1.6&	0.0264&	0.0239&	0.0242&	0.0261&	0.0264&	0.0242&	0.0210&	0.0233&	0.0246&	0.0279\\
		1.7&	0.0209&	0.0244&	0.0242&	0.0247&	0.0225&	0.0215&	0.0227&	0.0260&	0.0225&	0.0229\\
		1.8&	0.0233&	0.0229&	0.0281&	0.0217&	0.0216&	0.0222&	0.0220&	0.0240&	0.0217&	0.0262\\
		1.9&	0.0242&	0.0202&	0.0275&	0.0235&	0.0255&	{\bf 0.0198}&	0.0227&	0.0221&	0.0238&	0.0257	\\ \bottomrule
			\end{tabular}
		}
	\end{table}
	
	In Table \ref{tb-cons},  we present the influences of  the positive square summable step size $\alpha_k=a/k$, where $a\in[0.1,1]$, and  the positive penalization parameter $\beta_k=bk/\|\mathbf{A}\|^2$, where $b\in[0.1,1.9]$. We can notice that almost all  computational runtimes  are between 0.02 and 0.04 seconds.  Moreover, we observe that for each choice of $\alpha_k$, the larger penalization parameter $\beta_k$ gives a better result, and in this experiment, the best result is obtained from the combination of $\alpha_k=0.6/k$ and $\beta_k=1.9k/\|\mathbf{A}\|^2$.
	
	In Table \ref{error--tb}, we show the averaged computational runtime and the average number of iterations  of 10 sampling for  several numbers of target sets $m$ and several dimensions $n$. For the sake of completeness, we also present the average value of norm $\|\mathbf{A}\|$.  We terminate Algorithm \ref{algorithm-ergodic-smooth} by the optimality tolerance $\epsilon=10^{-6}$. 
	
		\begin{table}[H]
		\centering
		\setlength{\tabcolsep}{12pt}
		\caption{\label{error--tb}Behavior of Algorithm \ref{algorithm-ergodic-smooth} on  the generalized Heron problems with consistent constraints.}
		{\scriptsize \begin{tabular}{@{}l   l  r  r r  r r  r r  r r @{}}
				\hline\noalign{\smallskip}
				$m$   	   &        $n$  	&  $\|\mathbf{A}\|$  	& Time  	 & \#(Iters)    			 \\ \midrule
				5		&2		&12.7429		&0.0292	&1210\\
			             &3		 &30.4313		&0.0331	&1367\\
			             &5		&92.6700	&0.0642	&2925\\
			             &10    &423.3893	&0.2367	&8822\\ 
			             &20	&1978.0400	&1.4004	&39735\\
			             &50 	&16519.4145&	26.9685	&444919\\ \midrule
			    10		& 2    &24.2202		&0.0485	&1305\\        
			             &3		&55.5478	&0.0879	   &2252\\
				         &5		&167.3363	&0.1785	&4097\\
				         &10	&721.7128	&0.8301	&13761\\
				         &20	&3269.1938	&4.7381	&42965\\
				         &50	&24146.8185&	40.2545	&196313\\ \midrule
				 20		&2     &47.3754		&0.1301	&2078\\       
				         &3		&108.9700	&0.2621	&3831\\
				         &5		&308.9152	&0.6069	&8240\\
				         &10	&1299.6429	&2.0165	&22944\\   
				         &20	&5502.7016	&8.7097	&62866\\
				       &50	&38131.4998&	94.9971	&239289\\ \midrule
				       50		&2     &116.4996		&0.7457	&4857\\        
				         &3		&264.7944	&1.4213	&8629\\
				         &5		&743.7431	&4.2728	&19139\\
			         	 &10	&3037.0681	&14.8371&	51719\\
			         	 &20	&12424.2950	&64.2897&	127730\\
			         	 &50	&81465.8917&	450.8698&	439601\\ \midrule
				100    &2     &232.7786		&3.4001	&8879 	 \\ 
						&3    	&524.9095	&7.2561	&18363 	 \\ 
			             &5	    &1466.0786	&17.5919	&37689 	 \\ 
				         &10	&5907.6284	&60.5538&	100652  	 \\ 
				         &20	&24024.5749	&322.0257&	251252\\ 
			             &50 & 153858.3795	&2918.1447	&741560 \\
				\hline\noalign{\smallskip}
			\end{tabular}
		}
	\end{table}

		As shown in Table \ref{error--tb}, for the same number of target sets $m$, we observe that the higher
		dimension needs a longer time and a higher number of iterations. In a similar direction, we  notice that for the same dimension $n$, the larger number of target sets takes considerably more computational time than the smaller one. In particular, we also observe that the computational time for the case when $m=100$ is almost a hundred times of that of the smaller case $m=10$.

		In the next experiments, we consider the generalized Heron problem with inconsistent constraint, that is, $\mathbf{b}$ is not a zero vector $\mathbf{0}_{\mathbb{R}^{m^2}}$.   We also use the number of target sets $m=5$, the dimension $n=2$, and perform $10$ samplings for the different randomly chosen matrix $\mathbf{A}\in \mathbb{R}^{25\times 5}$, the balls $C_1,\ldots, C_5\subset\mathbb{R}^2$, the starting point $x_1\in\mathbb{R}^2$, and the vector $\mathbf{b}$ in the interval $(0,1)$. The influences of corresponding parameters $\alpha_k$ and $\beta_k$ in Algorithm \ref{algorithm-ergodic-smooth} are obtained when Algorithm \ref{algorithm-ergodic-smooth} is terminated by the optimal tolerance $10^{-6}$ and its runtimes, which are shown in Table  \ref{tb-incons} are averaged. 
		\begin{table}[H]
			\centering
			\setlength{\tabcolsep}{6pt}
			{\scriptsize		\caption{\label{tb-incons} Comparison of  algorithm runtime   for  different choices of step sizes  $\alpha_k = a/k$ and penalization parameter $\beta_k=bk/\|\mathbf{A}\|^2$ in  the case of inconsistent constraint.}
				\begin{tabular}{@{}l   r r  r r  r r  r r  r r  r r r r r r r r r@{}}
					\toprule  $a$ $\rightarrow$  & \multirow{2}{*}{$0.1$} & \multirow{2}{*}{$0.2$} & \multirow{2}{*}{$0.3$} & \multirow{2}{*}{$0.4$} &  \multirow{2}{*}{$0.5$} & \multirow{2}{*}{$0.6$}  & \multirow{2}{*}{$0.7$}  & \multirow{2}{*}{$0.8$} & \multirow{2}{*}{$0.9$} & \multirow{2}{*}{$1$} \\
					
					$b$  $\downarrow$  	  	 \\ \midrule
					0.1&	0.0975&	0.0953&	0.0861&	0.0863&	0.0725&	0.0862&	0.0692&	0.0908&	0.0906&	0.0759\\
					0.2&	0.0632&	0.0558&	0.0564&	0.0614&	0.0463&	0.0523&	0.0616&	0.0567&	0.0500&	0.0588\\
					0.3&	0.0510&	0.0563&	0.0492&	0.0570&	0.0422&	0.0493&	0.0389&	0.0582&	0.0455&	0.0444\\
					0.4&	0.0417&	0.0393&	0.0415&	0.0395&	0.0463&	0.0443&	0.0464&	0.0478&	0.0403&	0.0379\\
					0.5&	0.0392&	0.0410&	0.0314&	0.0415&	0.0403&	0.0446&	0.0395&	0.0388&	0.0394&	0.0510\\
					0.6&	0.0381&	0.0373&	0.0395&	0.0311&	0.0383&	0.0315&	0.0355&	0.0361&	0.0369&	0.0307\\
					0.7&	0.0368&	0.0343&	0.0313&	0.0339&	0.0361&	0.0313&	0.0346&	0.0341&	0.0286&	0.0390\\
					0.8&	0.0353&	0.0293&	0.0304&	0.0318&	0.0300&	0.0326&	0.0306&	0.0331&	0.0318&	0.0267\\
					0.9&	0.0296&	0.0280&	0.0298&	0.0319&	0.0306&	0.0279&	0.0320&	0.0274&	0.0275&	0.0287\\
					1.0&	0.0282&	0.0276&	0.0305&	0.0282&	0.0323&	0.0292&	0.0302&	0.0303&	0.0270&	0.0254\\
					1.1&	0.0269&	0.0273&	0.0320&	0.0277&	0.0271&	0.0252&	0.0344&	0.0257&	0.0264&	0.0297\\
					1.2&	0.0284&	0.0203&	0.0278&	0.0246&	0.0199&	0.0274&	0.0272&	0.0219&	0.0273&	0.0237\\
					1.3&	0.0242&	0.0233&	0.0275&	0.0283&	0.0267&	0.0249&	0.0262&	0.0273&	0.0267&	0.0252\\
					1.4&	0.0294&	0.0257&	0.0215&	0.0251&	0.0237&	0.0214&	0.0299&	0.0240&	0.0241&	0.0233\\
					1.5&	0.0245&	0.0275&	0.0227&	0.0250&	0.0254&	0.0228&	0.0248&	0.0228&	0.0192&	0.0252\\
					1.6&	0.0241&	0.0234&	0.0225&	0.0191&	0.0235&	0.0215&	0.0254&	0.0210&	0.0248&	0.0213\\
					1.7&	0.0233&	0.0248&	0.0220&	0.0261&	0.0199&	0.0211&	0.0235&	0.0219&	0.0208&	0.0233\\
					1.8&	0.0224&	0.0214&	0.0227&	0.0200&	0.0207&	0.0255&	0.0238&	0.0210&	0.0178&	{\bf 0.0175}\\
					1.9&	0.0251&	0.0210&	0.0201&	0.0194&	0.0212&	0.0234&	0.0254&	0.0205&	0.0224&	0.0188	\\ \bottomrule
				\end{tabular}
			}
		\end{table}
	
		We present in Table \ref{tb-incons} the influences of  the positive step size $\alpha_k=a/k$, where $a\in[0.1,1]$, and  the penalization parameter $\beta_k=bk/\|\mathbf{A}\|^2$, where $b\in[0.02,0.04]$. In the same manner with the consistent case, we can notice that  for all choices of step size $\alpha_k$,  a smaller penalization parameter $\beta_k$ considerably requires more computational runtime, and in this experiment, the best result is obtained from the combination of $\alpha_k=1/k$ and $\beta_k=1.8k/\|\mathbf{A}\|^2$.

			In a similar fashion as the consistent case,  we show the average computational runtime and the average number of iterations  of 10 sampling for  several numbers of target sets $m$ and several dimensions $n$ for the inconsistent case in Table \ref{error--tb-incons}. In this experiment, we also terminate Algorithm \ref{algorithm-ergodic-smooth} by the optimality tolerance $\epsilon=10^{-6}$. 
		\begin{table}[H]
			\centering
			\setlength{\tabcolsep}{12pt}
			\caption{\label{error--tb-incons}Behavior of Algorithm \ref{algorithm-ergodic-smooth} on   generalized Heron problems with inconsistent constraint.}
		{\scriptsize \begin{tabular}{@{}l   l  r  r r  r r  r r  r r @{}}
					\hline\noalign{\smallskip}
					$n$   	   &        $m$  	&  $\|\mathbf{A}\|$  	& Time  	 & \#(Iters)    			 \\ \midrule
			5	&2	&12.4861	&0.0295	&1112\\
			     &3	&30.5263	&0.0323	&1333\\
			     &5	&98.5510	 &0.0402	&1605\\
			 	&10	&430.1475	&0.1857	&6707\\
				&20	&1989.9481	&1.3728	&37874\\
				&50	&16358.4264	&27.6430	&446870\\\midrule
		   10	&2	&24.5911	&0.0486	&1319\\
				&3 &55.6295	&0.0556	&1379\\			
				&5	&163.4317	&0.0792	&1759\\
				&10	&725.6254	&0.3125	&4874\\
				&20	&3173.0835	&1.8285	&17425\\
				&50	&23579.2386	&19.7460	&94644\\ \midrule
			20 &2	&46.6140	&0.0770	&1197\\		
				&3	&107.6329	&0.1053	&1567\\
				&5  &309.4062	&0.1414	&1918\\
				&10	&1304.4604	&0.4817	&5267\\
				&20	&5521.4899	&1.9743	&14588\\
				&50	&38699.0252	&26.7447	&67902\\ \midrule
			50&2 	&116.5993	&0.3056	&2017\\
				&3	&263.8920	&0.2797	&1713\\
				&5	&742.5804	&0.6016	&2612\\
				&10	&3024.8978	&2.1776	&7343\\
				&20	&12455.7174	&8.2004	&16805\\
				&50	&81713.0491	&65.7634	&64529\\ \midrule
			100	&2 &232.2942	&0.8234	&2064\\ 
			   &3&524.2420	&0.9983	&2367\\  
				&5&1466.3736	&1.5053	&2991\\ 
               &10&5918.4690	&5.5408	&8583\\ 
               &20&23967.7064	&24.1894&21119\\  
               &50&	153872.0238&	251.9113	&74525\\
					\hline\noalign{\smallskip}
				\end{tabular}
			}
		\end{table}
	
		In Table \ref{error--tb-incons}, we notice that for the same number of target sets $m$,  almost all cases of higher dimension $m$ need a longer  runtime and a higher number of iterations. 	This is also similar manner in the context of the number of iterations. Moreover, we notice that for the same number of target  sets and dimension, the generalized Heron problem with inconsistent constraint requires less iterations and computational runtime compared with the consistent case.



		
			\section{Conclusions}
			We consider the splitting method called  the incremental proximal gradient method with a penalty term for solving a minimization problem of the sum of a finite number of convex functions subject to the set of minimizers of a convex differentiable  function. The advantage of our method is that it allows us not only to compute the proximal operator or the gradient of each function separately but also to consider a general sense of the constrained set. Under some suitable assumptions, we show the convergence of iterates to an optimal solution. Finally, we propose some numerical experiments on  the image inpainting problem and the generalized Heron problems.

			
			\section*{Acknowledgement}
				The authors are thankful to two anonymous referees and the Associate Editor for comments
				and remarks which improved the quality and presentation  of the paper. The authors are also thankful to Professor Radu Ioan Bo\c{t} for his suggestion on image inpainting problem. This work is  supported by the Thailand Research Fund under the Project RAP61K0012.

			
			\section*{Appendix}
			\subsection*{A. Key Tool Lemmas}\label{proof}
		This subsection is dedicated to the proofs of the important lemmas relating to the sequence generated by Algorithm \ref{algorithm-ergodic-smooth}.
		
		\begin{lemma}\label{IPGP-lemma-12}
			Let $u\in\mathcal{S}$ and $p\in N_{\arg\min g}(u)$ be such that $0=p+\sum_{i=1}^mv_i+\sum_{i=1}^m\nabla h_i(u)$, where $v_i\in\partial f_i(u)$ for all $i=1,\ldots,m$. Then  for every $k\geq1$ and $\eta>0$, we have
			\begin{eqnarray*}
				&&\|x_{k+1}-u\|^2-\|x_{k}-u\|^2+\frac{\eta}{1+\eta}\alpha_k\beta_k g(x_{k})+\left(1-\frac{\eta}{1+\eta}\right)\sum_{i=1}^m\|\varphi_{i+1,k}-\varphi_{i,k}\|^2\\
				&\leq&\alpha_k\left(\frac{2(1+\eta)}{\eta}\alpha_k-\frac{2}{\max_{1\leq i\leq m}L_i}\right)\sum_{i=1}^m\|\nabla h_i(\varphi_{i,k})-\nabla h_i(u)\|^2\\
				&&+\left(\left(1+\frac{\eta}{2(1+\eta)}\right)\alpha_k\beta_k-\frac{2}{L_g(1+\eta)}\right)\alpha_k\beta_k\|\nabla g(x_k)\|^2\\
				&&+\frac{2m(m+1)(1+\eta)}{\eta}\alpha_k^2\sum_{i=1}^m\|\nabla h_i(u)+v_i\|^2\\
				&&+\frac{\eta}{1+\eta}\alpha_k\beta_k\left[g^*\left(\frac{2p}{\frac{\eta}{1+\eta}\beta_k}\right)-\sigma_{\arg\min g}\left(\frac{2p}{\frac{\eta}{1+\eta}\beta_k}\right)\right].
			\end{eqnarray*}
		\end{lemma}
		\begin{proof} Let $k\geq1$ be fixed. For all $i=1,\ldots,m$, it follows from the definition of proximity operator $\alpha_k f_i$ that $\varphi_{i,k}-\alpha_k\nabla h_i(\varphi_{i,k})-\varphi_{i+1,k}\in\alpha_k\partial f_i(\varphi_{i+1,k})$. Since $v_i\in\partial f_i(u)$, the monotonicity of $\partial f_i$ implies that 
			\begin{eqnarray*}\<\varphi_{i,k}-\alpha_k\nabla h_i(\varphi_{i,k})-\varphi_{i+1,k}-\alpha_k v_i,\varphi_{i+1,k}-u\>\geq0,
			\end{eqnarray*}
			or  equivalently,
			\begin{eqnarray*}\<\varphi_{i,k}-\varphi_{i+1,k},u-\varphi_{i+1,k}\>\leq \alpha_k\<\nabla h_i(\varphi_{i,k})+v_i,u-\varphi_{i+1,k}\>.
			\end{eqnarray*}
			This implies that
			\begin{eqnarray}\label{lemma-eqn-1}\|\varphi_{i+1,k}-u\|^2-\|\varphi_{i,k}-u\|^2+\|\varphi_{i+1,k}-\varphi_{i,k}\|^2\leq 2\alpha_k\<\nabla h_i(\varphi_{i,k})+ v_i,u-\varphi_{i+1,k}\>.
			\end{eqnarray}
			
			Summing up the inequalities (\ref{lemma-eqn-1}) for all $i=1,\ldots,m$, we obtain
			\begin{eqnarray}\label{lemma-eqn-2}\|\varphi_{m+1,k}-u\|^2-\|\varphi_{1,k}-u\|^2&+&\sum_{i=1}^m\|\varphi_{i+1,k}-\varphi_{i,k}\|^2\nonumber\\
			&\leq& 2\alpha_k\sum_{i=1}^m\<\nabla h_i(\varphi_{i,k})+ v_i,u-\varphi_{i+1,k}\>\nonumber\\
			&=& 2\alpha_k\sum_{i=1}^m\<\nabla h_i(\varphi_{i,k})- \nabla h_i(u),u-\varphi_{i,k}\>\nonumber\\
			&&+ 2\alpha_k\sum_{i=1}^m\<\nabla h_i(\varphi_{i,k})- h_i(u),\varphi_{i,k}-\varphi_{i+1,k}\>\nonumber\\
			&&	+ 2\alpha_k\sum_{i=1}^m\<\nabla h_i(u)+v_i,u-\varphi_{i+1,k}\>.
			\end{eqnarray}
			
			Let us consider the first term in the right-hand side of (\ref{lemma-eqn-2}). For each $i=1,\ldots,m$, we note that $\nabla h_i$ is $\frac{1}{L_i}$-cocoercive, that is,
			\begin{eqnarray*}
				\<\nabla h_i(\varphi_{i,k})- \nabla h_i(u),\varphi_{i,k}-u\>\geq\frac{1}{L_i}\|\nabla h_i(\varphi_{i,k})-\nabla h_i(u)\|^2,
			\end{eqnarray*}
			and so
			\begin{eqnarray}\label{lemma-eqn-3}
			2\alpha_k\sum_{i=1}^m\<\nabla h_i(\varphi_{i,k})- \nabla h_i(u),u-\varphi_{i,k}\>
			&\leq&-2\alpha_k \sum_{i=1}^m\frac{1}{L_i}\|\nabla h_i(\varphi_{i,k})-\nabla h_i(u)\|^2\nonumber\\
			&\leq& \frac{-2\alpha_k}{\max_{1\leq i\leq m}L_i}\sum_{i=1}^m\|\nabla h_i(\varphi_{i,k})-\nabla h_i(u)\|^2.
			\end{eqnarray}
			
			For the second term of the right-hand side of (\ref{lemma-eqn-2}), for each $i=1,\ldots,m$, we note that
			\begin{eqnarray*}
				2\alpha_k\<\nabla h_i(\varphi_{i,k})- h_i(u),\varphi_{i,k}-\varphi_{i+1,k}\>&\leq& \frac{2(1+\eta)}{\eta}\alpha_k^2\|\nabla h_i(\varphi_{i,k})-\nabla h_i(u)\|^2\\
				&&+\frac{\eta}{2(1+\eta)}\|\varphi_{i,k}-\varphi_{i+1,k}\|^2,
			\end{eqnarray*}
			which yields
			\begin{eqnarray}\label{lemma-eqn-4}
			2\alpha_k\sum_{i=1}^m\<\nabla h_i(\varphi_{i,k})- h_i(u),\varphi_{i,k}-\varphi_{i+1,k}\>&\leq& \frac{2(1+\eta)}{\eta}\alpha_k^2\sum_{i=1}^m\|\nabla h_i(\varphi_{i,k})-\nabla h_i(u)\|^2\nonumber\\
			&&+\frac{\eta}{2(1+\eta)}\sum_{i=1}^m\|\varphi_{i,k}-\varphi_{i+1,k}\|^2.
			\end{eqnarray}

			Substituting (\ref{lemma-eqn-3}) and (\ref{lemma-eqn-4}) in (\ref{lemma-eqn-2}), we obtain
			\begin{eqnarray}\label{lemma-eqn-4-2}\|\varphi_{m+1,k}-u\|^2-\|\varphi_{1,k}-u\|^2&+&\left(1-\frac{\eta}{2(1+\eta)}\right)\sum_{i=1}^m\|\varphi_{i+1,k}-\varphi_{i,k}\|^2\nonumber\\
			&\leq& 
			\alpha_k\left(\frac{2(1+\eta)}{\eta}\alpha_k-\frac{2}{\max_{1\leq i\leq m}L_i}\right)\sum_{i=1}^m\|\nabla h_i(\varphi_{i,k})-\nabla h_i(u)\|^2\nonumber\\
			&&	+ 2\alpha_k\sum_{i=1}^m\<\nabla h_i(u)+v_i,u-\varphi_{i+1,k}\>.
			\end{eqnarray}
			
			Now, for the last term in the right-hand side of (\ref{lemma-eqn-4-2}), we have
			\begin{eqnarray}\label{lemma-eqn-5}
			2\alpha_k\sum_{i=1}^m\<\nabla h_i(u)+v_i,x_{k}-\varphi_{i+1,k}\>&\leq& \frac{2m(m+1)(1+\eta)}{\eta}\alpha_k^2\sum_{i=1}^m\|\nabla h_i(u)+v_i\|^2\nonumber\\
			&&+\frac{\eta}{2m(m+1)(1+\eta)}\sum_{i=1}^m\|x_{k}-\varphi_{i+1,k}\|^2.
			\end{eqnarray}
			
			Now, for each $i=1,\ldots,m$,  the triangle inequality yields
			\begin{eqnarray}\label{lemma-eqn-5-2}
			\|x_{k}-\varphi_{i+1,k}\|&\leq&\|x_{k}-\varphi_{1,k}\|+\sum_{j=1}^{i}\|\varphi_{j+1,k}-\varphi_{j,k}\|\\
			&\leq&\|x_{k}-\varphi_{1,k}\|+\sum_{i=1}^m\|\varphi_{i+1,k}-\varphi_{i,k}\|,\nonumber
			\end{eqnarray}
			and so
			\begin{eqnarray*}
				\|x_{k}-\varphi_{i+1,k}\|^2&\leq&\left(\|x_{k}-\varphi_{1,k}\|+\sum_{i=1}^m\|\varphi_{i+1,k}-\varphi_{i,k}\|\right)^2\\
				&\leq&(m+1)\left(\|x_{k}-\varphi_{1,k}\|^2+\sum_{i=1}^m\|\varphi_{i+1,k}-\varphi_{i,k}\|^2\right).
			\end{eqnarray*}
			Summing up the above inequalities for all $i=1,\ldots,m$, we have
			\begin{eqnarray*}
				\sum_{i=1}^m\|x_{k}-\varphi_{i+1,k}\|^2	&\leq&m(m+1)\left(\|x_{k}-\varphi_{1,k}\|^2+\sum_{i=1}^m\|\varphi_{i+1,k}-\varphi_{i,k}\|^2\right).
			\end{eqnarray*}
			
			Multiplying this inequality by $\frac{\eta}{2m(m+1)(1+\eta)}$, the inequality (\ref{lemma-eqn-5}) becomes
			\begin{eqnarray}\label{lemma-eqn-6}
			2\alpha_k\sum_{i=1}^m\<\nabla h_i(u)+v_i,x_{k}-\varphi_{i+1,k}\>&\leq& \frac{2m(m+1)(1+\eta)}{\eta}\alpha_k^2\sum_{i=1}^m\|\nabla h_i(u)+v_i\|^2\nonumber\\
			&&+\frac{\eta}{2(1+\eta)}\|x_{k}-\varphi_{1,k}\|^2+\frac{\eta}{2(1+\eta)}\sum_{i=1}^m\|\varphi_{i+1,k}-\varphi_{i,k}\|^2,\nonumber\\
			\end{eqnarray}
			and then
			\begin{eqnarray}\label{lemma-eqn-12-2}
			2\alpha_k\sum_{i=1}^m\<\nabla h_i(u)+v_i,u-\varphi_{i+1,k}\>
			&\leq& 2\alpha_k\sum_{i=1}^m\<\nabla h_i(u)+v_i,u-x_k\>\nonumber\\
			&&+2\alpha_k\sum_{i=1}^m\<\nabla h_i(u)+v_i,x_k-\varphi_{i+1,k}\>\nonumber\\
			&\leq& 2\alpha_k\sum_{i=1}^m\<\nabla h_i(u)+v_i,u-x_k\>\nonumber\\
			&&\frac{2m(m+1)(1+\eta)}{\eta}\alpha_k^2\sum_{i=1}^m\|\nabla h_i(u)+v_i\|^2	\nonumber\\
			&&+\frac{\eta}{2(1+\eta)}\|x_{k}-\varphi_{1,k}\|^2	+\frac{\eta}{2(1+\eta)}\sum_{i=1}^m\|\varphi_{i+1,k}-\varphi_{i,k}\|^2.\nonumber
			\end{eqnarray}
			
			Hence (\ref{lemma-eqn-4-2}) becomes
			\begin{eqnarray}\label{lemma-eqn-12-3}&&\|\varphi_{m+1,k}-u\|^2-\|\varphi_{1,k}-u\|^2+\left(1-\frac{\eta}{(1+\eta)}\right)\sum_{i=1}^m\|\varphi_{i+1,k}-\varphi_{i,k}\|^2\nonumber\\
			&\leq& 	\alpha_k\left(\frac{2(1+\eta)}{\eta}\alpha_k-\frac{2}{\max_{1\leq i\leq m}L_i}\right)\sum_{i=1}^m\|\nabla h_i(\varphi_{i,k})-\nabla h_i(u)\|^2\nonumber\\
			&&	+ \frac{2m(m+1)(1+\eta)}{\eta}\alpha_k^2\sum_{i=1}^m\|\nabla h_i(u)+v_i\|^2+ \frac{\eta}{2(1+\eta)}\alpha_k^2\beta_k^2\|\nabla g(x_k)\|^2\nonumber\\ &&2\alpha_k\sum_{i=1}^m\<\nabla h_i(u)+v_i,u-x_{k}\>.
			\end{eqnarray}
			
			On the other hand, from the definition of $\varphi_{1,k}$, we note that
			\begin{eqnarray}\label{lemma-eqn-12-4}
			\|\varphi_{1,k}-u\|^2&=&\|\varphi_{1,k}-x_k\|^2+\|x_k-u\|^2+2\<\varphi_{1,k}-x_k,x_k-u\>\nonumber\\
			&=&\alpha_k^2\beta_k^2\|\nabla g(x_k)\|^2+\|x_k-u\|^2-2\alpha_k\beta_k\<\nabla g(x_k),x_k-u\>.
			\end{eqnarray}
			
			Furthermore, since $\nabla g$ is $\frac{1}{L_g}$-cocoercive and $\nabla g(u)=0$, we have
			\begin{eqnarray*}
				2\alpha_k\beta_k\<\nabla g(x_k),x_k-u\>
				&=&2\alpha_k\beta_k\<\nabla g(x_k)-\nabla g(u),x_k-u\>\nonumber\\
				&\geq&\frac{2\alpha_k\beta_k}{L_g}\|\nabla g(x_k)-\nabla g(u)\|^2=\frac{2\alpha_k\beta_k}{L_g}\|\nabla g(x_k)\|^2,
			\end{eqnarray*}
			which implies that
			\begin{eqnarray}\label{lemma-eqn-12-5}
			-2\alpha_k\beta_k\<\nabla g(x_k),x_k-u\>
			&\leq&-\frac{2\alpha_k\beta_k}{L_g}\|\nabla g(x_k)\|^2.
			\end{eqnarray}
			
			Moreover, since $g$ is convex and $g(u)=0$, we have
			\begin{eqnarray}\label{lemma-eqn-12-6}
			-2\alpha_k\beta_k\<\nabla g(x_k),x_k-u\>	&=& 2\alpha_k\beta_k\<\nabla g(x_k),u-x_k\>\nonumber\\
			&\leq& 2\alpha_k\beta_k\left(g(u)-g(x_k)\right)=-2\alpha_k\beta_kg(x_k).
			\end{eqnarray}
			
			Combining (\ref{lemma-eqn-12-5}) and (\ref{lemma-eqn-12-6}), we obtain
			\begin{eqnarray*}
				-2\alpha_k\beta_k\<\nabla g(x_k),x_k-u\>	
				&\leq& -\frac{2\alpha_k\beta_k}{L_g(1+\eta)}\|\nabla g(x_k)\|^2-\frac{2\eta}{1+\eta}\alpha_k\beta_kg(x_k).
			\end{eqnarray*}
			
			From this inequality, together with the inequality (\ref{lemma-eqn-12-4}) and the definition of $x_{k+1}$, it follows that
			\begin{eqnarray}\label{lemma-eqn-12-7}&&\|x_{k+1}-u\|^2-\|x_k-u\|^2+\left(1-\frac{\eta}{(1+\eta)}\right)\sum_{i=1}^m\|\varphi_{i+1,k}-\varphi_{i,k}\|^2+\frac{\eta}{1+\eta}\alpha_k\beta_kg(x_k)\nonumber\\
			&\leq& 	\alpha_k\left(\frac{2(1+\eta)}{\eta}\alpha_k-\frac{2}{\max_{1\leq i\leq m}L_i}\right)\sum_{i=1}^m\|\nabla h_i(\varphi_{i,k})-\nabla h_i(u)\|^2\nonumber\\
			&&+\left(\left(1+\frac{\eta}{2(1+\eta)}\right)\alpha_k\beta_k-\frac{2}{L_g(1+\eta)}\right)\alpha_k\beta_k\|\nabla g(x_k)\|^2\nonumber\\
			&&	+ \frac{2m(m+1)(1+\eta)}{\eta}\alpha_k^2\sum_{i=1}^m\|\nabla h_i(u)+v_i\|^2\nonumber\\ 
			&&2\alpha_k\sum_{i=1}^m\<\nabla h_i(u)+v_i,u-x_{k}\>-\frac{\eta}{1+\eta}\alpha_k\beta_kg(x_k).
			\end{eqnarray}

			In particular, since $p\in\argmin g$, we have
			{\small		\begin{eqnarray*}
					2\alpha_k\sum_{i=1}^m\<\nabla h_i(u)+v_i,u-x_{k}\>-\frac{\eta}{1+\eta}\alpha_k\beta_kg(x_k)
					&=&-2\alpha_k\<p,u-x_{k}\>-\frac{\eta}{1+\eta}\alpha_k\beta_kg(x_{k})\\
					&=&2\alpha_k\<p,x_{k}\>-\frac{\eta}{1+\eta}\alpha_k\beta_kg(x_{k})-2\alpha_k\<p,u\>\\
					&=&\frac{\eta}{1+\eta}\alpha_k\beta_k\left[\left\langle\frac{2p}{\frac{\eta}{1+\eta}\beta_k},x_{k}\right\rangle-g(x_{k})-\left\langle\frac{2p}{\frac{\eta}{1+\eta}\beta_k},u\right\rangle\right]\\
					&\leq&\frac{\eta}{1+\eta}\alpha_k\beta_k\left[g^*\left(\frac{2p}{\frac{\eta}{1+\eta}\beta_k}\right)-\left\langle\frac{2p}{\frac{\eta}{1+\eta}\beta_k},u\right\rangle\right]\\
					&=&\alpha_k\beta_k\left[g^*\left(\frac{2p}{\frac{\eta}{1+\eta}\beta_k}\right)-\sigma_{\argmin g}\left(\frac{2p}{\frac{\eta}{1+\eta}\beta_k},u\right)\right].
				\end{eqnarray*}
			}
			Combining this relation and (\ref{lemma-eqn-12-7}),  the required inequality is finally obtained.					
		\end{proof}
		
					The following proposition also plays an essential role in  convergence analysis.  
		\begin{proposition}\label{lemma-alvarez} \cite{P87}
			Let $(a_k)_{k\geq1}$, $(b_k)_{k\geq1}$, and $(c_k)_{k\geq1}$ be real sequences. Assume that $(a_k)_{k\geq1}$ is bounded from below, $(b_k)_{k\geq1}$ is nonnegative, $\sum_{k=1}^\infty c_k<+\infty$, and 
			$$a_{k+1}-a_k+b_k\leq c_k \indent \forall k \geq 1.$$
			Then the sequence  $(a_k)_{k\geq1}$ converges and $\sum_{k=1}^\infty b_k<+\infty$.
		\end{proposition}

		The following lemma is a collection of  some convergence properties of the sequences involved in our analysis.
		
		\begin{lemma}\label{key-lemma4-2} The following statements hold:
			
			(i) The sequence $(x_k)_{k\geq1}$ is  quasi-Fej\'{e}r monotone relative to $\mathcal{S}$.
			
			(ii) For each $u\in \mathcal{S}$, the limit $\lim_{k\to+\infty}\|x_k-u\|$ exists. Moreover, we have $\sum_{k=1}^\infty\alpha_k\beta_kg(x_{k})<+\infty$, 
			$\sum_{k=1}^\infty\alpha_k\beta_k\|\nabla g(x_{k})\|^2<+\infty$, and $\sum_{k=1}^\infty\sum_{i=1}^m\|\varphi_{i+1,k}-\varphi_{i,k}\|^2<+\infty$.
			
			(iii) $\lim_{k\to+\infty}g(x_{k})=\lim_{k\to+\infty}\|\nabla g(x_{k})\|=\lim_{k\to+\infty}\sum_{i=1}^m\|\varphi_{i+1,k}-\varphi_{i,k}\|^2=\lim_{k\to+\infty}\|x_k-\varphi_{1,k}\|=0$.
			
			(iv) Every sequential cluster point of the sequence $(x_k)_{k\geq1}$ lies in $\argmin g$.
		\end{lemma}
		\begin{proof} (i) Since $\limsup_{k\to+\infty}\alpha_k\beta_k<\frac{2}{L_g}$, there exists $k_0\in\N$ such that $\alpha_k\beta_k<\frac{2}{L_g}$
			for all $k\geq k_0$.                  Now, by picking $\eta_0\in\left(0,\frac{2\left(\frac{2}{L_g}-\limsup_{k\to+\infty}\alpha_k\beta_k\right)}{3\limsup_{k\to+\infty}\alpha_k\beta_k}\right)$, we have  $\alpha_k\beta_k<\frac{4}{L_g(2+3\eta_0)}$ for all $k\geq k_0$. Hence, there must exist $M>0$ such that for all $k\geq k_0$, we have
			$$\alpha_k\beta_k < M < \frac{4}{L_g(2+3\eta_0)}=\frac{2}{L_g(1+\eta_0)\left(1+\frac{\eta_0}{2(1+\eta_0)}\right)}.$$
			On the other hand, since $\alpha_k\to 0$, there exists $k_1\in\N$ such that 
			$$\frac{2(1+\eta_0)}{\eta_0}\alpha_k-\frac{2}{\max_{1\leq i\leq m}L_i}<0, \indent \forall k\geq k_1.$$
			Let $u\in \mathcal{S}$ be given.  	For every $k\geq \max\{k_0,k_1\}$, we have from Lemma \ref{IPGP-lemma-12}  that
			\begin{eqnarray}\label{eqn-vip}
			&&\|x_{k+1}-u\|^2-\|x_{k}-u\|^2+\frac{\eta_0}{1+\eta_0}\alpha_k\beta_k g(x_{k})+\left(1-\frac{\eta_0}{1+\eta_0}\right)\sum_{i=1}^m\|\varphi_{i+1,k}-\varphi_{i,k}\|^2\nonumber\\
			&&+\left(\frac{2}{L_g(1+\eta_0)}-\left(1+\frac{\eta_0}{2(1+\eta_0)}\right)M\right)\alpha_k\beta_k\|\nabla g(x_k)\|^2\nonumber\\
			&\leq&\frac{2m(m+1)(1+\eta_0)}{\eta_0}\alpha_k^2\sum_{i=1}^m\|\nabla h_i(u)+v_i\|^2\nonumber\\
			&&+\frac{\eta_0}{1+\eta_0}\alpha_k\beta_k\left[g^*\left(\frac{2p}{\frac{\eta_0}{1+\eta_0}\beta_k}\right)-\sigma_{\arg\min g}\left(\frac{2p}{\frac{\eta_0}{1+\eta_0}\beta_k}\right)\right].
			\end{eqnarray}
			Since  the right-hand side is summable and the last three terms of the left-hand side are nonnegative, the sequence $(x_k)_{k\geq1}$ is  quasi-Fej\'{e}r monotone relative to $\mathcal{S}$.
			
			(ii)	Since the right-hand side of (\ref{eqn-vip}) is summable, the statement in (ii) follows immediately from Proposition \ref{lemma-alvarez}.
			
			(iii)	By (ii), it is obvious that $\lim_{k\to+\infty}\|x_{k}-\varphi_{1,k}\|^2=\lim_{k\to+\infty}\sum_{i=1}^m\|\varphi_{i+1,k}-\varphi_{i,k}\|^2=0$.
			Moreover, by the assumption (H3), we also have $\lim_{k\to+\infty}\|\nabla g(x_{k})\|=\lim_{k\to+\infty}g(x_{k})=0$.
			
			(iv) Let $w$ be a sequential cluster point of $(x_k)_{k\geq1}$ and $(x_{k_j})_{j\geq1}$ be a subsequence of $(x_k)_{k\geq1}$ such that $x_{k_j}\to w$. By the lower semicontinuity of $g$, we obtain
			$$g(w)\leq \liminf_{j\to+\infty}g(x_{k_j})=\lim_{k\to+\infty}g(x_{k})=0,$$
			and so $w\in\arg\min g$.
		\end{proof}

		\subsection*{B. Parameters Combinations for PGM and FISTA}
			In this subsection, we present the ISNR values performed by the classical proximal-gradient method (PGM) and FISTA for various parameters combinations.
				\begin{table}[H]
				\centering
				\setlength{\tabcolsep}{6pt}
				{\scriptsize		\caption{\label{tb-isnr-fista} ISNR values performed by FISTA after 20 iterations for different choices of parameters $\lambda_1$ and $\lambda_2$.}
					\begin{tabular}{@{}l   r r  r r  r r  r r  r r  r r r r r r r r r@{}}
						\toprule  $\lambda_1$ $\rightarrow$  & \multirow{2}{*}{$0.001$} & \multirow{2}{*}{$0.005$} & \multirow{2}{*}{$0.01$} & \multirow{2}{*}{$0.05$} &  \multirow{2}{*}{$0.1$} & \multirow{2}{*}{$0.2$}  & \multirow{2}{*}{$0.3$}  & \multirow{2}{*}{$0.4$} & \multirow{2}{*}{$0.5$} \\
						
						$\lambda_2$  $\downarrow$  	  	 \\ \midrule
					$10^{-8}$   &	1.1003	&5.0197	&9.3646	&15.9509	&15.1522	&13.1999	&11.6012	&10.2787	&9.1532\\
					$10^{-7}$   &	1.1003	&5.0197	&9.3646	&15.9509	&15.1522	&13.1999	&11.6012	&10.2787	&9.1532\\
					$10^{-6}$   &	1.1003	&5.0196	&9.3643	&15.9509	&15.1521	&13.1998	&11.6011	&10.2786	&9.1531\\
						$10^{-5}$&	1.0999	&5.0182	&9.3623	&15.9509	&15.1517	&13.1994	&11.6007	&10.2783	&9.1528\\
						$10^{-4}$&	1.0966	&5.0049	&9.3414	&15.9511	&15.1479	&13.1950	&11.5966	&10.2745	&9.1495\\
						$10^{-3}$&	1.0639	&4.8724	&9.1253	&15.9440	&15.1076	&13.1512	&11.5551	&10.2372	&9.1162\\
						0.005	    & 0.9296  &4.3075  &8.1022	&15.7625	&14.9024	&12.9478  & 11.3680	&10.0706	&8.9685\\
						0.01	      & 0.7854	&3.6751	 &6.9042  &15.3534	 &14.5746	&12.6772	& 11.1297	&9.8619	    &8.7847		\\ \bottomrule
					\end{tabular}
				}
			\end{table}
		In Table \ref{tb-isnr-fista}, we observe that the combination $\lambda_1=0.05$ with $\lambda_2=10^{-4}$ leads to the largest ISNR values of 15.9511 dB.
		
		\begin{table}[H]
			\centering
			\setlength{\tabcolsep}{6pt}
			{\scriptsize		\caption{\label{tb-isnr-pgm-1} ISNR values performed by PGM after 20 iterations for different choices of parameters $\lambda_1$ and $\lambda_2$ with $\gamma=1/(\lambda_2+1)$.}
				\begin{tabular}{@{}l   r r  r r  r r  r r  r r  r r r r r r r r r@{}}
					\toprule  $\lambda_1$ $\rightarrow$  & \multirow{2}{*}{$0.001$} & \multirow{2}{*}{$0.005$} & \multirow{2}{*}{$0.01$} & \multirow{2}{*}{$0.05$} &  \multirow{2}{*}{$0.1$} & \multirow{2}{*}{$0.2$}  & \multirow{2}{*}{$0.3$}  & \multirow{2}{*}{$0.4$} & \multirow{2}{*}{$0.5$} \\
					
					$\lambda_2$  $\downarrow$  	  	 \\ \midrule
					$10^{-8}$ &	0.1896	&0.9280	&1.8158	&7.9777	&13.3422&	13.1484	&11.6114	&10.2972&	9.1721\\
					$10^{-7}$ &0.1896	&0.9280	&1.8158	&7.9777	&13.3422&	13.1484	&11.6114	&10.2972&	9.1721\\
					$10^{-6}$ &	0.1896	&0.9280	&1.8158	&7.9776	&13.3421&	13.1483	&11.6113	&10.2971&	9.1721\\
					$10^{-5}$ &	0.1896	&0.9279	&1.8156	&7.9769	&13.3412&	13.1479	&11.6109	&10.2967&	9.1717\\
					$10^{-4}$ &	0.1894	&0.9270	&1.8140	&7.9695	&13.3318&	13.1433	&11.6068	&10.2930&	9.1684\\
					$10^{-3}$ &	0.1876	&0.9184	 &1.7974 &7.8963&	13.2378&	13.0972&	11.5653&	10.2556&	9.1349\\
					0.005	     & 0.1799	&0.8816	&1.7262	&7.5816	&12.8122	&12.8849	&11.3779	&10.0885	&8.9863\\
					0.01           & 0.1707	 &0.8381 &1.6423  &7.2120&	12.2729	&12.6055	&11.1393	&9.8788	&8.8013	\\ \bottomrule
				\end{tabular}
			}
		\end{table}
		In Table \ref{tb-isnr-pgm-1}, we observe that the combinations $\lambda_1=0.1$ with $\lambda_2=10^{-7}, 10^{-8}$ lead to the largest ISNR values of 13.3422 dB.  
		
		\begin{table}[H]
			\centering
			\setlength{\tabcolsep}{6pt}
			{\scriptsize		\caption{\label{tb-isnr-pgm-2} ISNR values performed by PGM after 20 iterations for different choices of parameters $\lambda_2$ and $\gamma$ with $\lambda_1=0.1$.}
				\begin{tabular}{@{}l   r r  r r  r r  r r  r r  r r r r r r r r r@{}}
					\toprule  $\gamma$ $\rightarrow$  & \multirow{2}{*}{$\frac{1}{(\lambda_2+1)}$} & \multirow{2}{*}{$\frac{1.1}{(\lambda_2+1)}$} & \multirow{2}{*}{$\frac{1.2}{(\lambda_2+1)}$} & \multirow{2}{*}{$\frac{1.3}{(\lambda_2+1)}$} &  \multirow{2}{*}{$\frac{1.4}{(\lambda_2+1)}$} & \multirow{2}{*}{$\frac{1.5}{(\lambda_2+1)}$}  & \multirow{2}{*}{$\frac{1.6}{(\lambda_2+1)}$}  & \multirow{2}{*}{$\frac{1.7}{(\lambda_2+1)}$} & \multirow{2}{*}{$\frac{1.8}{(\lambda_2+1)}$}  & \multirow{2}{*}{$\frac{1.9}{(\lambda_2+1)}$} \\
					
					$\lambda_2$  $\downarrow$  	  	 \\ \midrule
					$10^{-8}$ &13.3422	&13.9381	&14.3418	&14.6012	&14.7664	&14.8744	&14.9489	&15.0031	&15.0458	&15.0842\\
					$10^{-7}$ &13.3422	&13.9381	&14.3418	&14.6012	&14.7664	&14.8744	&14.9489	&15.0031	&15.0458	&15.0842\\
					$10^{-6}$ &13.3421	&13.9380	&14.3417	&14.6012	&14.7663	&14.8744	&14.9489	&15.0030	&15.0458	&15.0842\\
					$10^{-5}$ &13.3412	&13.9371	&14.3409	&14.6005	&14.7657	&14.8738	&14.9483	&15.0025	&15.0453	&15.0837\\
					$10^{-4}$ &13.3318	&13.9284	&14.3331	&14.5936	&14.7595	&14.8680	&14.9429	&14.9973	&15.0402	&15.0787\\
					$10^{-3}$ &13.2378	&13.8399	&14.2537	&14.5231	&14.6960	&14.8091	&14.8871	 &14.9437	 &14.9883	 &15.0278\\
					$0.005$   &12.8122	&13.4286	&13.8748	&14.1807	&14.3841	 &14.5192	  &14.6117	   &14.6783	    &14.7302	&14.7738\\
					$0.01$      &12.2729  &12.8882	  &13.3562	  &13.6962	  &13.9347	  &14.0987	 &14.2119	   &14.2927	    &14.3537	 &14.4025	\\ \bottomrule
				\end{tabular}
			}
		\end{table}
		In Table \ref{tb-isnr-pgm-2}, we observe that the combinations $\lambda_2=10^{-6}, 10^{-7}, 10^{-8}$ with $\gamma=1.9/(\lambda_2+1)$ lead to the largest ISNR values of 15.0842 dB.  
		

	\end{document}